%% file: hammock_4.tex
\documentclass[a4paper]{amsart}
\input{preamble}
\begin{document}
\title{Hammock localization via Segal animae}
\begin{abstract}
We give a short, conceptual account of Dwyer--Kan's hammock localization in the setting of $\infty$-categories.
Starting from Mazel-Gee's formula for localization via the relative Rezk nerve, we show that its Segalification can be described explicitly as a Segal anima of zig-zags.
We also compute its mapping animae and recover and generalize Dwyer--Kan's hammock formula.
As an application, we show that, when a relative $\infty$-category supports fractions, the mapping animae of localization admit a simple description.
This gives a unifying treatment for several formulas of this type in the existing literature.
\end{abstract}

\author{Kensuke Arakawa}
\address{K.A.: Research Institute for Mathematical Sciences, Kyoto University, 606-8502, Kyoto, Japan}
\author{Bastiaan Cnossen}
\address{B.C.: Fakult\"{a}t f\"{u}r Mathematik, Universit\"{a}t Regensburg, 93040 Regensburg, Germany}
\maketitle

\setcounter{tocdepth}{1}
\tableofcontents

\section*{Introduction}

Dwyer and Kan's hammock localization is an explicit simplicial category model for the $\pr{\infty,1}$-localization of a relative category \cite{DK80_2}.
Historically, it was one of the first settings in which mapping animae, rather than mapping sets, arose naturally from a category with weak equivalences, anticipating one of the central ideas of modern $\pr{\infty,1}$-category theory (henceforth $\infty$-category theory).
Its explicit combinatorial description of mapping animae, modeled as simplicial sets, has also proved to be an indispensable technical tool and has played a crucial role in theory and applications of higher category theory \cite{BM11,Raventos_Hammock_Localization,MG18,KS19,Pav25,BCGS26,Rov26,RS26}.

Although the construction is explicit, its original simplicial-set formulation requires extensive combinatorial arguments.
Our aim is to give a short, conceptual account in $\infty$-categorical terms, while retaining the explicit description in terms of zig-zags.

To explain our main result, we recall the following elegant formula for localization, due to Rezk \cite{Rez01} and Mazel-Gee \cite{MG19}.
A \emph{relative $\infty$-category} is an $\infty$-category $\cat C$ equipped with a wide subcategory $\cat W\subset\cat C$.
Associated to such is the \emph{relative Rezk nerve} $\N^{\rel}\pr{\cat C,\cat W}$, a simplicial anima defined by
\[
\N^{\rel}\pr{\cat C,\cat W}_{n}=\abs{\Fun\pr{[n],\cat C}\times_{\cat C^{n+1}}\cat W^{n+1}}
\]
where $\abs -\from\Cat_{\infty}\to\An$ denotes the left adjoint to the inclusion.
In the case $\cat W$ consists of the equivalences of $\cat C$, this reduces to the ordinary Rezk nerve, hence the name.
By a theorem of Mazel-Gee \cite{MG19}, the complete Segalification of $\N^{\rel}\pr{\cat C,\cat W}$ presents the localization of $\cat C$ at $\cat W$.

Complete Segalification of simplicial animae is a two-step process: First Segalify, then complete.
The main theorem of this paper asserts that zig-zags arise quite naturally in the first step.
In more detail, we define a Segal anima $\cat Z_{\cat C,\cat W}$ of zig-zags in $\pr{\cat C,\cat W}$ (Definition \ref{def:Z_CW}).
By definition, its anima of $1$-simplices is the colimit, over all finite zig-zag shapes, of the animae of diagrams
\[
\bullet\xleftarrow{\sim}\bullet\to\bullet\cdots\bullet\xleftarrow{\sim}\bullet\to\bullet
\]
whose left-pointing arrows must lie in $\cat W$.
We then show that:

\begin{customthm}[Theorem {\ref{thm:Z_Segal}}]\label{thm:main1}

$\cat Z_{\cat C,\cat W}$ is the Segalification of $\N^{\rel}\pr{\cat C,\cat W}$.

\end{customthm}

Since Rezk completion leaves mapping animae unchanged, Theorem \ref{thm:main1} and the aforementioned theorem of Mazel-Gee tell us that the mapping animae of $\cat Z_{\cat C,\cat W}$ are equivalent to those of $\cat C[\cat W^{-1}]$.
We compute the former and obtain the following:

\begin{customthm}[Theorem {\ref{thm:Z_Segal}} and Corollary {\ref{cor:hammock_anima_is_anima_of_loc}}]\label{thm:main2}

The mapping animae of $\cat Z_{\cat C,\cat W}$ (and hence of $\cat C[\cat W^{-1}]$) are given by
\[
\cat Z_{\cat C,\cat W}\pr{x,y}\simeq\cat C[\cat W^{-1}]\pr{x,y}\simeq\Hamm_{\cat C,\cat W}\pr{x,y},
\]
where the right-hand side is the hammock anima of Definition \ref{def:hammock}, defined as the colimit, over all finite zig-zag shapes, of the animae of zig-zags in $\cat C$ from $x$ to $y$.

\end{customthm}

When $\cat C$ is a $1$-category, Theorem \ref{thm:main2} recovers the homotopy type of the mapping simplicial set of Dwyer--Kan's hammock localization.
The same formula has also been used for relative $\infty$-categories, for example in \cite{BCGS26}, but to our knowledge a proof in this generality has not previously appeared.
Together, Theorems \ref{thm:main1} and \ref{thm:main2} offer an unexpected, beautiful connection between two a priori rather different presentations of localizations: the relative Rezk nerve and Dwyer--Kan's hammock localization.

The anima $\Hamm_{\cat C,\cat W}\pr{x,y}$ is conceptually appealing, but it can be difficult to analyze because it involves zig-zags of arbitrary length.
For relative $1$-categories satisfying ``homotopy calculus of (left, right, or two-sided) fractions,'' Dwyer and Kan observed that this anima can be simplified substantially.
We will generalize their definitions to the $\infty$-categorical setting and obtain the following corollary:

\begin{customcor}[Theorem \ref{thm:fraction}]\label{cor:main3}

If $\pr{\cat C,\cat W}$ supports left fractions, then $\cat C[\cat W^{-1}]\pr{x,y}$ is the anima of zig-zags of the form $x\to\bullet\xleftarrow{\sim}y$.
More precisely,
\[
\cat C[\cat W^{-1}]\pr{x,y}\simeq\abs{\cat C_{x/}\times_{\cat C}\cat W_{y/}}.
\]
There are similar results for right or two-sided fractions.

\end{customcor}

Several classes of relative categories are known to support fractions, such as Brown's categories of fibrant objects \cite{Bro73} and categories admitting Gabriel--Zisman's calculus of fractions \cite{GZ67}.
These relative categories have $\infty$-categorical generalizations, due to Cisinski \cite{HCHA} (for the former) and Carranza--Kapulkin--Lindsey \cite{CKL25} (for the latter).
We will show that $\infty$-categories of fibrant objects support fractions, as do relative $\infty$-categories admitting quasicategorical calculus of left fractions (Theorems \ref{thm:CLF} and \ref{thm:fibrationcats}).
This will exhibit Corollary \ref{cor:main3} as a unifying source for several formulas for mapping animae of localizations that had previously been established by separate arguments.

\subsection*{Related work}

As mentioned above, our approach to hammock localization is based on Segal animae, in contrast to the original approach via categories enriched over simplicial sets.
Nevertheless, the two approaches are compatible in the sense to be explained in Section \ref{sec:DKcomparison}.
In particular, we can extract the main results of \cite{DK80_2} from our results.

In \cite{MG18}, Mazel-Gee constructs another model of hammock localization using what he calls $\sAn$-enriched $\infty$-categories.
The identification of its associated $\infty$-category with the localization is not established there.
In Section \ref{sec:MGcomparison}, we compare Mazel-Gee's construction with ours and deduce this identification in Corollary \ref{cor:MGhamloc_is_localization}.

A special case of Corollary \ref{cor:main3} for model $\infty$-categories plays an essential role in the proof of the main result of \cite{MG21}, and it is proved by a more intricate argument in loc. cit.

There are already several variations of calculus of fractions, such as \cite{CKL25} and \cite[\S 7.2]{HCHA}.
We will show that a relative $\infty$-category $\pr{\cat C,\cat W}$ satisfying the calculus of fractions of \cite{CKL25} supports fractions in our sense (but not conversely; see Remark \ref{rem:CLFgeneral_converse_false}).
Combining this with Corollary \ref{cor:main3} gives an alternative proof of the mapping-anima formula of \cite[Corollary 10.10]{CKL25}.
On the other hand, as observed in \cite{ACK25}, calculi in the sense of \cite[\S 7.2]{HCHA} describe mapping animae using ``resolutions'' of objects.
Their hypotheses are of a different nature from ours, and we do not know an implication in either direction. 

\subsection*{Organization of the paper}

This paper has five sections and two appendices.

The first three sections are the core of the paper.
In Section \ref{sec:Zigzag}, we state and prove the main theorem (Theorem \ref{thm:main1}).
We then give a sufficient condition for the hammock formula to be simplified in Section \ref{sec:fraction}, and provide further criteria and examples of relative $\infty$-categories satisfying this condition in Section \ref{sec:criteria}. 

The next two sections discuss the compatibility of our construction with existing ones.
In Section \ref{sec:DKcomparison}, we compare our construction with Dwyer--Kan's hammock localization \cite{DK80_2}.
In Section \ref{sec:MGcomparison}, we compare our construction with Mazel-Gee's hammock localization \cite{MG18}.

The appendices contain several standard results in $\infty$-category theory.

\subsection*{Notation and conventions}

We write $\Cat_{\infty}$ for the $\infty$-category of small $\infty$-categories, and $\An$ for its full subcategory of animae (i.e., $\infty$-groupoids).
We say that a subcategory of an $\infty$-category is \emph{wide} if it contains all equivalences.
We write $\RelCat_{\infty}\subset\Fun\pr{[1],\Cat_{\infty}}$ for the full subcategory spanned by the essentially surjective monomorphisms, i.e., wide subcategory inclusions.
A typical object $\cat W\hookrightarrow\cat C$ in this $\infty$-category will be denoted by $\pr{\cat C,\cat W}$.

In Subsections \ref{subsec:quasi} and \ref{subsec:fibrant}, we will use lemmas whose existing proofs rely critically on combinatorics of quasicategories (Lemmas \ref{lem:filtered} and \ref{lem:4.3.15}).
The results in the remaining (sub)sections use only very formal properties of $\infty$-categories and are not bound to a specific model of $\infty$-categories.

\subsection*{AI disclosure:} Coding agents (Claude/Codex) were used in the revision of the manuscript, which also led us to a simpler proof of Lemma~\ref{lem:chi_final}.

\section{\label{sec:Zigzag}The Segal anima \texorpdfstring{$\protect\cat Z_{\protect\cat C,\protect\cat W}$}{Z(C,W)} of zig-zags}

The goal of this section is to associate to each relative $\infty$-category $\pr{\cat C,\cat W}$ a simplicial anima $\cat Z_{\cat C,\cat W}$ of ``zig-zags'' in $\pr{\cat C,\cat W}$ (Definition \ref{def:Z_CW}), and characterize it as the Segalification of the relative Rezk nerve of $\pr{\cat C,\cat W}$ (Theorem \ref{thm:Z_Segal}).
As a corollary, we obtain a description of the mapping animae of the localization $\cat C[\cat W^{-1}]$ as hammock animae (Corollary \ref{cor:hammock_anima_is_anima_of_loc}).

This section is organized as follows: We state the main theorem in Subsection \ref{subsec:def_state}, together with all definitions needed for the statement.
The proof of the theorem will be divided into three parts (Subsections \ref{subsec:Z_Segal}, \ref{subsec:hammock}, and \ref{subsec:Segalification}) and will be completed in Subsection \ref{subsec:summary}.

\subsection{\label{subsec:def_state}Definitions and statement of the main theorem}

In this subsection, we define the simplicial anima $\cat Z_{\cat C,\cat W}$ and related constructions.
We then state the main theorem of this paper (Theorem \ref{thm:Z_Segal}).

The definition of $\cat Z_{\cat C,\cat W}$ is based on the following variations of the simplex category $\Del$.

\begin{defn}
We define the category $\Zig$ of \textit{zig-zags} as follows: Its objects are pairs $[S,T]$, where $S,T$ are disjoint sets of positive integers such that $S\cup T=\{1,\dots,\abs S+\abs T\}$.
(We allow the case $S=T=\emptyset$.)
We will identify such pairs with relative categories freely generated by morphisms $f_{s}\from s-1\to s$ for $s\in S$ and $f_{t}\from t\to t-1$ for $t\in T$, where the $f_{t}$'s are weak equivalences.
For instance, $[\{1,2,5\},\{3,4\}]$ looks like the following:
\[
0\to1\to2\xleftarrow{\sim}3\xleftarrow{\sim}4\to5.
\]
The morphisms of $\Zig$ are relative functors that preserve the natural ordering of integers.

We also write $\Zig_{\partial}\subset\Zig$ for the subcategory spanned by the morphisms that preserve the boundary elements (i.e., the minimal and the maximal elements for the natural ordering of integers).
The category $\Zig_{\partial}$ has a monoidal structure $\vee$ given by ``concatenation'': 
\[
[S,T]\vee[S',T']=[S\cup\{\abs S+\abs T+i\mid i\in S'\},T\cup\{\abs S+\abs T+t\mid t\in T'\}].
\]
We will write $[n]=[\{1,\dots,n\},\emptyset]$ and $[-n]=[\emptyset,\{1,\dots,n\}]$.
We also use notations such as $[k_{1},\dots,k_{n}]=[k_{1}]\vee\dots\vee[k_{n}]=\bigvee^{n}_{i=1}[k_{i}]$ for $k_{1},\dots,k_{n}\in\mathbb{Z}$, and $[n,Z]=[n]\vee Z$ and $ZZ'=[Z,Z']=Z\vee Z'$ for $Z,Z'\in\Zig_{\partial}$ and $n\in\mathbb{Z}$.
\end{defn}

\begin{notation}
We write $l\from\Zig\to\Del$ for the functor that sends $Z=[S,T]$ to $l\pr Z=[\abs S]$.
The canonical map $q_Z\from Z\to l\pr Z$ sends $i$ to $\abs{S\cap\{1,\dots,i\}}$ and exhibits $l\pr Z$ as the localization of $Z$ at its weak equivalences.
On morphisms, $l$ is the functor induced by these localizations.
For $j\in l\pr Z$, we write $Z_j=q_Z^{-1}\pr j$ for the corresponding fiber.
\end{notation}

\begin{rem}
The category $\Zig_{\partial}$ is the opposite of the category $\mathbf{II}$ introduced in \cite[\S 4]{DK80_2}.
\end{rem}

\begin{defn}
\label{def:int_Zig}We write $\pi\from\int\Zig^{\bullet}_{\partial}\to\Del$ for the cartesian fibration associated to the monoidal category $\pr{\Zig_{\partial},\vee}$.
Thus, its objects are tuples $\pr{Z_{1},\dots,Z_{n}}$, where $n\geq0$ and $Z_{1},\dots,Z_{n}\in\Zig_{\partial}$.
A morphism $\pr{Z_{1},\dots,Z_{n}}\to\pr{Z'_{1},\dots,Z'_{m}}$ consists of a morphism $u\from[n]\to[m]$ in $\Del$ and maps $\{Z_{i}\to\bigvee_{u\pr{i-1}<j\leq u\pr i}Z'_{j}\}_{1\leq i\leq n}$ in $\Zig_{\partial}$.

Note that the assignment $[n]\mapsto\underset{n\text{ times}}{(\underbrace{[1],\dots,[1]})}$ determines a section $\iota\from\Del\to\int\Zig^{\bullet}_{\partial}$ of this cartesian fibration.
\end{defn}

We can now define $\cat Z_{\cat C,\cat W}$ and the hammock anima.
\begin{defn}
\label{def:Z_CW}Let $\pr{\cat C,\cat W}$ be a relative $\infty$-category.
For each $Z\in\Zig$, we define the $\infty$-category of zig-zags of shape $Z$ in $\pr{\cat C,\cat W}$ by
\[
\Zig_{Z}\pr{\cat C,\cat W}=\Fun_{\rel}\pr{Z,\pr{\cat C,\cat W}}\times_{\cat C^{\mathrm{ob}Z}}\cat W^{\mathrm{ob}Z}.
\]
The assignment $Z\mapsto\abs{\Zig_{Z}\pr{\cat C,\cat W}}$ determines a presheaf $\abs{\Zig_{\bullet}\pr{\cat C,\cat W}}\in\PSh\pr{\Zig}$, and we define a simplicial anima $\cat Z_{\cat C,\cat W}$ as the left Kan extension depicted as:
\[\begin{tikzcd}[column sep =large]
	{\pr{\int\Zig^{\bullet}_{\partial}}^{\op}} && {\Zig^\op} && \An \\
	{\Del^\op}
	\arrow[""{name=0, anchor=center, inner sep=0}, "{(Z_1,\dots,Z_n)\mapsto Z_1\cdots Z_n}", from=1-1, to=1-3]
	\arrow[from=1-1, to=2-1]
	\arrow["{|\Zig_{\bullet}(\cat{C},\cat{W})|}", from=1-3, to=1-5]
	\arrow[""{name=1, anchor=center, inner sep=0}, "{\cat{Z}_{\cat{C},\cat{W}}}"', dashed, from=2-1, to=1-5]
	\arrow[between={0.2}{0.8}, Rightarrow, from=0, to=1]
\end{tikzcd}\]

Since the projection $\pr{\int\Zig^{\bullet}_{\partial}}^{\op}\to\Del^{\op}$ is a cocartesian fibration, the pointwise formula for left Kan extension reduces to fiberwise colimit \cite[\href{https://kerodon.net/tag/02ZM}{Tag 02ZM}]{kerodon}.
That is, 
\[
\pr{\cat Z_{\cat C,\cat W}}_{n}\simeq\colim_{\pr{Z_{1},\dots,Z_{n}}\in\pr{\Zig^{n}_{\partial}}^{\op}}\abs{\Zig_{Z_{1}\cdots Z_{n}}\pr{\cat C,\cat W}}.
\]
In particular, $\pr{\cat Z_{\cat C,\cat W}}_{0}\simeq\abs{\cat W}$.
\end{defn}

\begin{rem}
Informally, a point of $\pr{\cat Z_{\cat C,\cat W}}_{n}$ is represented by shapes $Z_{1},\dots,Z_{n}$ together with a zig-zag of concatenated shape $Z_{1}\cdots Z_{n}$ in $\pr{\cat C,\cat W}$, or equivalently by $n$ composable finite zig-zags with their $n-1$ junctions marked.
Paths between such representatives are generated, coherently, by natural transformations that are pointwise in $\cat W$ and by changes of the individual shapes $Z_i$, where the latter identify representatives obtained by inserting identity morphisms or composing consecutive morphisms pointing in the same direction within any one of the $n$ constituent zig-zags.
\end{rem}

Note that restricting the top composite of the diagram along $\iota^{\op}$ gives the relative Rezk nerve. Consequently, restricting the unit of the defining left Kan extension along $\iota^{\op}$ gives a comparison map
\[
\gamma_{\cat C,\cat W}\from\N^{\rel}\pr{\cat C,\cat W}\to\cat Z_{\cat C,\cat W}.
\]

\begin{defn}
\label{def:hammock}Let $\pr{\cat C,\cat W}$ be a relative $\infty$-category.
Given a pair of objects $x,y\in\cat C$, we define the \textit{hammock anima} $\Hamm_{\cat C,\cat W}\pr{x,y}$ from $x$ to $y$ by
\[
\Hamm_{\cat C,\cat W}\pr{x,y}=\colim_{Z\in\Zig^{\op}_{\partial}}\abs{\Zig_{Z}\pr{\cat C,\cat W}_{x,y}},
\]
where $\Zig_{Z}\pr{\cat C,\cat W}_{x,y}=\Zig_{Z}\pr{\cat C,\cat W}\times_{\cat W\times\cat W}\{\pr{x,y}\}$.
\end{defn}

\begin{rem}
	When $\pr{\cat C,\cat W}$ is a relative $1$-category, we will identify $\Hamm_{\cat C,\cat W}\pr{x,y}$ with the anima associated to the mapping simplicial set $L^H\pr{\cat C,\cat W}\pr{x,y}$ in Dwyer--Kan's hammock localization; see Section \ref{sec:DKcomparison}.
\end{rem}

For $x,y \in \cat C$, the mapping anima $\cat Z_{\cat C,\cat W}\pr{x,y}$ of $\cat Z_{\cat C,\cat W}$ is given by
\[
\cat Z_{\cat C,\cat W}\pr{x,y}\simeq\pr{\colim_{Z\in\Zig^{\op}_{\partial}}\abs{\Zig_{Z}\pr{\cat C,\cat W}}}\times_{\abs{\cat W\times\cat W}}\{\pr{x,y}\}.
\]
The canonical comparison from the colimit of the fibers over $\pr{x,y}$ to the fiber over $\pr{x,y}$ of the colimit gives a map $\Hamm_{\cat C,\cat W}\pr{x,y}\to\cat Z_{\cat C,\cat W}\pr{x,y}$.

The following is the main result of this section:

\begin{thm}
\label{thm:Z_Segal}Let $\pr{\cat C,\cat W}$ be a relative $\infty$-category.

\begin{enumerate}
\item The simplicial anima $\cat Z_{\cat C,\cat W}$ is Segal.
\item For every pair of objects $x,y\in\cat C$, the map
\[
\Hamm_{\cat C,\cat W}\pr{x,y}\to\cat Z_{\cat C,\cat W}\pr{x,y}
\]
is an equivalence.
\item The map $\gamma_{\cat C,\cat W}\from\N^{\rel}\pr{\cat C,\cat W}\to\cat Z_{\cat C,\cat W}$ exhibits $\cat Z_{\cat C,\cat W}$ as a Segalification of $\N^{\rel}\pr{\cat C,\cat W}$.
\end{enumerate}
\end{thm}

\begin{rem}
Under the equivalence in part (2), composition in $\cat Z_{\cat C,\cat W}$ is induced by concatenation of zig-zags.
\end{rem}

By Mazel-Gee's theorem on the relative Rezk nerve (Theorem \ref{thm:univ_Rezk}), the left adjoint $\ac\from\Fun\pr{\Del^{\op},\An}\to\Cat_{\infty}$ of the Rezk nerve carries $\N^{\rel}\pr{\cat C,\cat W}$ to $\cat C[\cat W^{-1}]$.
Since the associated category functor $\ac$ does not affect mapping animae \cite[Theorem 7.7]{Rez01}, Theorem \ref{thm:Z_Segal} gives the following corollary:

\begin{cor}
\label{cor:hammock_anima_is_anima_of_loc}For every relative $\infty$-category $\pr{\cat C,\cat W}$, one has
\[
\ac\pr{\cat Z_{\cat C,\cat W}}\simeq\cat C[\cat W^{-1}].
\]
In particular, for every pair of objects $x,y\in\cat C$, we have
\[
\Hamm_{\cat C,\cat W}\pr{x,y}\simeq\cat C[\cat W^{-1}]\pr{x,y}.
\]
\end{cor}

The rest of this section is devoted to the proof of Theorem \ref{thm:Z_Segal}.
We prove part (1) in Subsection \ref{subsec:Z_Segal}, part (2) in Subsection \ref{subsec:hammock}, and part (3) in Subsection \ref{subsec:Segalification}.
We then summarize the discussion in Subsection \ref{subsec:summary}.

\subsection{\label{subsec:Z_Segal}\texorpdfstring{$\protect\cat Z_{\protect\cat C,\protect\cat W}$}{Z(C,W)} is a Segal anima}

In this subsection, we show that $\cat Z_{\cat C,\cat W}$ is a Segal anima (Proposition \ref{prop:segal}).

\begin{notation}
Given a relative $\infty$-category $\pr{\cat C,\cat W}$, we define a new $\infty$-category $\Zig_{\partial}\pr{\cat C,\cat W}\in\Cat_{\infty}$ by
\[
\Zig_{\partial}\pr{\cat C,\cat W}=\colim_{Z\in\Zig^{\op}_{\partial}}\Zig_{Z}\pr{\cat C,\cat W}.
\]
Since realization preserves colimits, there is a canonical equivalence
\[
\abs{\Zig_{\partial}\pr{\cat C,\cat W}}\simeq\pr{\cat Z_{\cat C,\cat W}}_{1}.
\]
The evaluation at the left and right endpoints determines functors $\dom,\codom\from\Zig_{\partial}\pr{\cat C,\cat W}\to\cat W$.
\end{notation}

We will repeatedly use the following elementary observation about colimits.

\begin{lem}
\label{lem:colim2}Let $f,g\from I\to\cat E$ be functors of $\infty$-categories, and let $\alpha\from f\to g$ be a natural transformation.
Suppose there are an endofunctor $e\from I\to I$ and natural transformations $\beta\from g\to fe$ and $\eta\from\id\to e$ such that $f\eta\simeq\beta\cdot\alpha$ and $g\eta\simeq\alpha e\cdot\beta$.
Then $\alpha$ induces an equivalence
\[
\colim_{I}f\xrightarrow{\simeq}\colim_{I}g,
\]
as long as the colimits exist.
\end{lem}

\begin{proof}
Let $\overline{\alpha}$ be the map in question, and let $\iota^{f}$ and $\iota^{g}$ denote the colimit cocones of $f$ and $g$.
Then an inverse equivalence $\overline{\beta}$ is induced by the cocone
\[
g\xrightarrow{\beta}fe\xrightarrow{\iota^{f}e}\colim_{I}f,
\]
as the assumed homotopies and the compatibility of the colimit cocones with $\eta$ give
\begin{align*}
\overline{\beta}\overline{\alpha}\cdot\iota^{f}&\simeq\iota^{f}e\cdot\beta\cdot\alpha\simeq\iota^{f}e\cdot f\eta\simeq\iota^{f},\\
\overline{\alpha}\overline{\beta}\cdot\iota^{g}&\simeq\iota^{g}e\cdot\alpha e\cdot\beta\simeq\iota^{g}e\cdot g\eta\simeq\iota^{g}. \qedhere
\end{align*}
\end{proof}

\begin{lem}
\label{lem:zigpartialfiber}Let $\pr{\cat C,\cat W}$ be a relative $\infty$-category.
For every morphism $w\from x\xrightarrow{\sim}y$ in $\cat W$, postcomposition and precomposition with $w$ induce equivalences
\begin{align*}
\abs{\Zig_{\partial}\pr{\cat C,\cat W}\times_{\cat W}\cat W_{/x}} & \xrightarrow{\simeq}\abs{\Zig_{\partial}\pr{\cat C,\cat W}\times_{\cat W}\cat W_{/y}},\\
\abs{\Zig_{\partial}\pr{\cat C,\cat W}\times_{\cat W}\cat W_{y/}} & \xrightarrow{\simeq}\abs{\Zig_{\partial}\pr{\cat C,\cat W}\times_{\cat W}\cat W_{x/}}.
\end{align*}
Here, the fiber products are formed using the codomain projection $\codom\from\Zig_{\partial}\pr{\cat C,\cat W}\to\cat W$.
\end{lem}

\begin{proof}
We will show that the first map is an equivalence; the proof for the second map is similar, and we indicate the necessary changes at the end of the proof.

Since the map $\cat W_{/x}\to\cat W$ is a right fibration, it is a Conduch\'e fibration, hence
\[
\hspace{-20pt}
\Zig_{\partial}\pr{\cat C,\cat W}\times_{\cat W}\cat W_{/x}\simeq\colim_{Z\in\Zig^{\op}_{\partial}}\pr{\Zig_{Z}\pr{\cat C,\cat W}\times_{\cat W}\cat W_{/x}}\simeq\colim_{Z\in\Zig^{\op}_{\partial}}\pr{\Zig'_{[Z,1]}\pr{\cat C,\cat W}_{x}};
\hspace{-20pt}
\]
here the subscript $x$ indicates taking the fiber over $x\in\cat W$ via the codomain functor, and $\Zig'_{[Z,1]}\pr{\cat C,\cat W}\subset\Zig_{[Z,1]}\pr{\cat C,\cat W}$ denotes the full subcategory of objects that send the rightmost arrow to a weak equivalence.

Define functors $f,g\from\Zig^{\op}_{\partial}\to\An$ by
\[
f\pr Z=\abs{\Zig'_{[Z,1]}\pr{\cat C,\cat W}_{x}},
\qquad
g\pr Z=\abs{\Zig'_{[Z,1]}\pr{\cat C,\cat W}_{y}},
\]
and let $\alpha\from f\to g$ be the natural transformation obtained by postcomposition with $w$. In light of the previous calculation, our goal is to show that the induced map $\overline{\alpha}\colon \colim f \to \colim g$ is an equivalence. We exhibit this as an instance of Lemma \ref{lem:colim2}. Define an endofunctor $e\from\Zig^{\op}_{\partial}\to\Zig^{\op}_{\partial}$ by $e\pr Z=[Z,1,-1]$.
Adjoining the detour $y\xleftarrow{w}x\xrightarrow{\id_x}x$ gives a natural transformation $\beta\from g\to fe$.
Collapsing the newly adjoined segment defines a natural transformation $e\to\id$ on $\Zig_{\partial}$, and hence a natural transformation $\eta\from\id\to e$ on $\Zig^{\op}_{\partial}$.
To apply Lemma \ref{lem:colim2}, it remains to construct equivalences
\[
f\eta\simeq\beta\cdot\alpha
\qquad\text{and}\qquad
g\eta\simeq\alpha e\cdot\beta.
\]
For the first equivalence, a zig-zag $\cdots z\xrightarrow[\sim]{u}x$ gives the following diagram naturally:
\[\begin{tikzcd}
	z & z & z & x \\
	z & y & x & x,
	\arrow[equal, from=1-1, to=1-2]
	\arrow[equal, from=1-1, to=2-1]
	\arrow["uw", from=1-2, to=2-2]
	\arrow[equal, from=1-3, to=1-2]
	\arrow["u", from=1-3, to=1-4]
	\arrow["u", from=1-3, to=2-3]
	\arrow[equal, from=1-4, to=2-4]
	\arrow["uw"', from=2-1, to=2-2]
	\arrow["w", from=2-3, to=2-2]
	\arrow[equal, from=2-3, to=2-4]
\end{tikzcd}\]
Here the top and bottom rows represent $f\eta$ and $\beta\cdot\alpha$, respectively; the identities in the top row precede $u$ because $f$ appends the final $[1]$.
As the zig-zag varies, the vertical maps give the desired homotopy.
Likewise, the second equivalence is induced by the following diagram, associated naturally to a zig-zag $\cdots z\xrightarrow[\sim]{u}y$:
\[\begin{tikzcd}
	z & z & z & y \\
	z & y & y & y \\
	z & y & x & y.
	\arrow[equal, from=1-1, to=1-2]
	\arrow[equal, from=1-1, to=2-1]
	\arrow["u"', from=1-2, to=2-2]
	\arrow[equal, from=1-3, to=1-2]
	\arrow["u", from=1-3, to=1-4]
	\arrow["u"', from=1-3, to=2-3]
	\arrow[equal, from=1-4, to=2-4]
	\arrow["u"', from=2-1, to=2-2]
	\arrow[equal, from=2-3, to=2-2]
	\arrow[equal, from=2-3, to=2-4]
	\arrow[equal, from=3-1, to=2-1]
	\arrow["u"', from=3-1, to=3-2]
	\arrow[equal, from=3-2, to=2-2]
	\arrow["w", from=3-3, to=2-3]
	\arrow["w", from=3-3, to=3-2]
	\arrow["w"', from=3-3, to=3-4]
	\arrow[equal, from=3-4, to=2-4]
\end{tikzcd}\]
As the zig-zag varies, this diagram determines a homotopy between $g\eta$ and $\alpha e\cdot\beta$.
Lemma \ref{lem:colim2} now shows that $\alpha$ induces an equivalence on colimits, proving the claim for the first map.

The proof for the second map is similar, using the endofunctor $e\from\Zig^{\op}_{\partial}\to\Zig^{\op}_{\partial}$ given by $e\pr Z=[Z,-1,1]$.
\end{proof}

\begin{prop}
\label{prop:segal}Let $\pr{\cat C,\cat W}$ be a relative $\infty$-category.
The simplicial anima $\cat Z_{\cat C,\cat W}$ is Segal.
\end{prop}

\begin{proof}
We must show that, for every $n\geq2$, the square 
\[\begin{tikzcd}
	{\operatorname{colim}_{(Z_i)_{i=1}^n\in(\mathrm{Zig}_\partial^n)^{\mathrm{op}}}|\mathrm{Zig}_{Z_1\cdots Z_n}(\mathcal{C},\mathcal{W})|} & {\operatorname{colim}_{Z_1\in\mathrm{Zig}_\partial^{\mathrm{op}}}|\mathrm{Zig}_{Z_1}(\mathcal{C},\mathcal{W})|} \\
	{\operatorname{colim}_{(Z_i)_{i=2}^{n}\in(\mathrm{Zig}_\partial^{n-1})^{\mathrm{op}}}|\mathrm{Zig}_{Z_2\cdots Z_n}(\mathcal{C},\mathcal{W})|} & {|\mathcal{W}|}
	\arrow[from=1-1, to=1-2]
	\arrow[from=1-1, to=2-1]
	\arrow["\codom", from=1-2, to=2-2]
	\arrow["\dom"', from=2-1, to=2-2]
\end{tikzcd}\]
is cartesian.
Fix $Z_{2},\dots,Z_{n}\in\Zig_{\partial}$ and write $Y=Z_{2}\cdots Z_{n}$. Observe that the domain projection
\[
q_Y=\dom\from\Zig_Y\pr{\cat C,\cat W}\to\cat W
\]
is a cartesian or cocartesian fibration: When $Y=[0]$, it is the identity of $\cat W$; if the leftmost arrow of $Y$ points to the right, cartesian transport is given by precomposition at the left endpoint; and if it points to the left, cocartesian transport is given by postcomposition at the left endpoint. Since cartesian and cocartesian fibrations are Conduch\'e fibrations, base change along $q_Y$ preserves colimits, and thus
\begin{align*}
\Zig_{\partial}\pr{\cat C,\cat W}\times_{\cat W}\Zig_Y\pr{\cat C,\cat W}
& \simeq\colim_{Z\in\Zig^{\op}_{\partial}}\pr{\Zig_Z\pr{\cat C,\cat W}\times_{\cat W}\Zig_Y\pr{\cat C,\cat W}}\\
& \simeq\colim_{Z\in\Zig^{\op}_{\partial}}\Zig_{ZY}\pr{\cat C,\cat W}.
\end{align*}

We claim that this pullback remains a pullback after realization. Indeed, this is an instance of Quillen's Theorem B, in the form of Corollary \ref{cor:ThmB} or its dual; the invariance hypothesis on the comma categories of $\codom\from\Zig_{\partial}\pr{\cat C,\cat W}\to\cat W$ is satisfied by Lemma \ref{lem:zigpartialfiber}. Since realization preserves colimits, we thus obtain a cartesian square of the form
\[\begin{tikzcd}
	{\operatorname{colim}_{Z\in \mathrm{Zig}_\partial^{\mathrm{op}}}|\mathrm{Zig}_{ZY}(\mathcal{C},\mathcal{W})|} & {\operatorname{colim}_{Z\in \mathrm{Zig}_\partial^{\mathrm{op}}}|\mathrm{Zig}_{Z}(\mathcal{C},\mathcal{W})|} \\
	{|\mathrm{Zig}_{Y}(\mathcal{C},\mathcal{W})|} & {|\mathcal{W}|.}
	\arrow[from=1-1, to=1-2]
	\arrow[from=1-1, to=2-1]
	\arrow["\codom", from=1-2, to=2-2]
	\arrow["\dom"', from=2-1, to=2-2]
\end{tikzcd}\]
Finally, take the colimit in the left column as $Z_{2},\dots,Z_{n}$ vary.
Colimits are universal in $\An$, so the resulting square remains cartesian.
Its left column is $\pr{\cat Z_{\cat C,\cat W}}_{n}\to\pr{\cat Z_{\cat C,\cat W}}_{n-1}$, while its right column is $\pr{\cat Z_{\cat C,\cat W}}_{1}\to\pr{\cat Z_{\cat C,\cat W}}_{0}$.
This is the desired Segal square.
\end{proof}

\subsection{\label{subsec:hammock}Mapping animae of \texorpdfstring{$\protect\cat Z_{\protect\cat C,\protect\cat W}$}{Z(C,W)} as hammock animae}

In this subsection, we show that the mapping animae of $\cat Z_{\cat C,\cat W}$ are computed by the hammock anima of Definition \ref{def:hammock} (Proposition \ref{prop:hammock}).

\begin{lem}
\label{lem:zigpartialfiber2}Let $\pr{\cat C,\cat W}$ be a relative $\infty$-category, and let $x'\xrightarrow{\sim}x$ and $y'\xrightarrow{\sim}y$ be two weak equivalences.
Precomposition with $x'\xrightarrow{\sim}x$ and $y'\xrightarrow{\sim}y$ induces an equivalence
\[
\abs{\pr{\cat W_{x/}\times\cat W_{y/}}\times_{\cat W\times\cat W}\Zig_{\partial}\pr{\cat C,\cat W}}\xrightarrow{\simeq}\abs{\pr{\cat W_{x'/}\times\cat W_{y'/}}\times_{\cat W\times\cat W}\Zig_{\partial}\pr{\cat C,\cat W}}.
\]
\end{lem}

\begin{proof}
The proof is entirely analogous to that of Lemma \ref{lem:zigpartialfiber}, using Lemma \ref{lem:colim2} with the endofunctor $e\from\Zig^{\op}_{\partial}\to\Zig^{\op}_{\partial}$ given by $e\pr Z=[-1,1,Z,-1,1]$.
\end{proof}

\begin{prop}
\label{prop:hammock}Let $\pr{\cat C,\cat W}$ be a relative $\infty$-category, and let $x,y\in\cat C$.
The canonical comparison
\[
\Hamm_{\cat C,\cat W}\pr{x,y}\to\cat Z_{\cat C,\cat W}\pr{x,y}
\]
is an equivalence.
\end{prop}

\begin{proof}
By the definition of the mapping anima, we must show that the square
\[\begin{tikzcd}
	{\operatorname{colim}_{Z\in\mathrm{Zig}_\partial^{\mathrm{op}}}|\mathrm{Zig}_{Z}(\mathcal{C},\mathcal{W})_{(x,y)}|} & {\operatorname{colim}_{Z\in\mathrm{Zig}_\partial^{\mathrm{op}}}|\mathrm{Zig}_{Z}(\mathcal{C},\mathcal{W})|} \\
	{\{(x,y)\}} & {|\mathcal{W}|\times |\mathcal{W}|}
	\arrow[from=1-1, to=1-2]
	\arrow[from=1-1, to=2-1]
	\arrow["{(\mathrm{dom},\mathrm{codom})}", from=1-2, to=2-2]
	\arrow[from=2-1, to=2-2]
\end{tikzcd}\]
is cartesian.
Consider the expanded diagram 
\[\begin{tikzcd}[column sep = tiny, scale cd = .75]
	{\colim_{Z\in\Zig_\partial^\op}|\Zig_{Z}(\cat{C},\cat{W})_{x,y}|} & {\colim_{Z\in\Zig_\partial^\op}|(\cat{W}_{x/}\times \cat{W}_{y/})\times_{\cat{W}\times \cat{W}}\Zig_Z(\cat{C},\cat{W})|} & {\colim_{Z\in\Zig_\partial^\op}|\Zig_{Z}(\cat{C},\cat{W})|} \\
	{\{(x,y)\}} & {|\cat{W}_{x/}\times \cat{W}_{y/}|} & {|\cat{W}\times \cat{W}|.}
	\arrow[from=1-1, to=1-2]
	\arrow[from=1-1, to=2-1]
	\arrow[from=1-2, to=1-3]
	\arrow[from=1-2, to=2-2]
	\arrow["{(\mathrm{dom},\mathrm{codom})}", from=1-3, to=2-3]
	\arrow[from=2-1, to=2-2]
	\arrow[from=2-2, to=2-3]
\end{tikzcd}\]
Since the projection $\cat W_{x/}\times\cat W_{y/}\to\cat W\times\cat W$ is a left fibration, it is a Conduch\'e fibration.
Thus, we have
\[
\hspace{-27pt}
\colim_{Z\in\Zig^{\op}_{\partial}}\pr{\cat W_{x/}\times\cat W_{y/}}\times_{\cat W\times\cat W}\Zig_{Z}\pr{\cat C,\cat W}\simeq\pr{\cat W_{x/}\times\cat W_{y/}}\times_{\cat W\times\cat W}\colim_{Z\in\Zig^{\op}_{\partial}}\Zig_{Z}\pr{\cat C,\cat W}.
\hspace{-27pt}
\]
It follows from Lemma \ref{lem:zigpartialfiber2} and Quillen's Theorem B (Corollary \ref{cor:ThmB}) that the right-hand square is cartesian.
Thus, it suffices to show that the left-hand square is cartesian.
Since $\cat W_{x/}\times\cat W_{y/}$ is weakly contractible (as it has an initial object), this is equivalent to the claim that the map
\[
\colim_{Z\in\Zig^{\op}_{\partial}}\abs{\Zig_{Z}\pr{\cat C,\cat W}_{x,y}}\to\colim_{Z\in\Zig^{\op}_{\partial}}\abs{\pr{\cat W_{x/}\times\cat W_{y/}}\times_{\cat W\times\cat W}\Zig_{Z}\pr{\cat C,\cat W}}
\]
is an equivalence. We apply Lemma \ref{lem:colim2}.
Let $f,g\from\Zig^{\op}_{\partial}\to\An$ denote the diagrams appearing in the source and target of this map, and let $\alpha\from f\to g$ be induced by the identity arrows at $x$ and $y$.
For the endofunctor $e\from\Zig^{\op}_{\partial}\to\Zig^{\op}_{\partial}$ given by $e\pr Z=[1,Z,-1]$, absorbing the two anchoring weak equivalences into the zig-zag gives a natural transformation $\beta\from g\to fe$.
Collapsing the newly adjoined outer segments defines a natural transformation $e\to\id$ on $\Zig_{\partial}$, and hence a natural transformation $\eta\from\id\to e$ on $\Zig^{\op}_{\partial}$.
The equivalence $f\eta\simeq\beta\cdot\alpha$ is immediate, while $g\eta\simeq\alpha e\cdot\beta$ is induced by the natural transformation whose outer components are the two anchoring weak equivalences and whose remaining components are identities.
Lemma \ref{lem:colim2} now shows that $\alpha$ induces an equivalence on colimits.
\end{proof}

We conclude this subsection with one of the consequences of Proposition \ref{prop:hammock}, which we will use in the coming subsections.

\begin{cor}
\label{cor:ham_equivalence}Let $\pr{\cat C,\cat W}$ be a relative $\infty$-category.
Any pair of morphisms $u\from x'\to x$ and $v\from y\to y'$ in $\cat W$ induces equivalences
\begin{align*}
\cat Z_{\cat C,\cat W}\pr{x,y}&\xrightarrow{\simeq}\cat Z_{\cat C,\cat W}\pr{x',y'},\\
\Hamm_{\cat C,\cat W}\pr{x,y}&\xrightarrow{\simeq}\Hamm_{\cat C,\cat W}\pr{x',y'},
\end{align*}
where the second map is induced by concatenating $u$ and $v$ on the left and right.
\end{cor}

\begin{proof}
The morphisms $u$ and $v$ determine a path from $\pr{x,y}$ to $\pr{x',y'}$ in $\abs{\cat W}\times\abs{\cat W}$, using the inverse of the path represented by $u$ in the first factor.
By definition of the mapping animae, transport along this path gives the first equivalence.
Under the equivalence of Proposition \ref{prop:hammock}, this transport map is represented by concatenating $u$ and $v$ at the two ends, giving the second equivalence.
\end{proof}

\subsection{\label{subsec:Segalification}\texorpdfstring{$\protect\cat Z_{\protect\cat C,\protect\cat W}$}{Z(C,W)} as a Segalification of the relative Rezk nerve}

In this subsection, we prove that $\cat Z_{\cat C,\cat W}$ is a Segalification of $\N^{\rel}\pr{\cat C,\cat W}$:

\begin{prop}
\label{prop:segalification}Let $\pr{\cat C,\cat W}$ be a relative $\infty$-category.
The map
\[
\gamma_{\cat C,\cat W}\from\N^{\rel}\pr{\cat C,\cat W}\to\cat Z_{\cat C,\cat W}
\]
of Definition \ref{def:Z_CW} exhibits $\cat Z_{\cat C,\cat W}$ as the Segalification of $\N^{\rel}\pr{\cat C,\cat W}$.
\end{prop}

The bulk of the work for Proposition \ref{prop:segalification} goes into establishing the case where $\pr{\cat C,\cat W}$ is some zig-zag $Z\in\Zig$. In this case, we write $\gamma_{Z}\colon \N^{\rel}\pr Z \to \cat Z_{Z}$ for the comparison map in question. A special feature in this situation is that the Segalification already happens to be complete.

\begin{lem}
\label{lem:MG_Z}For every $Z\in\Zig$, the map $\N^{\rel}\pr Z\to\N^{\rel}\pr{l\pr Z,l\pr Z^{\simeq}}\simeq\Delta^{l\pr Z}$ is a complete Segalification.
\end{lem}

\begin{proof}
This follows from Mazel-Gee's theorem on the relative Rezk nerve (Theorem \ref{thm:univ_Rezk}).
\end{proof}

\begin{lem}
\label{lem:Segalification_NrelZ}For every $Z\in\Zig$, the Segalification of $\N^{\rel}\pr Z$ is complete.
\end{lem}

\begin{proof}
Let $\eta_{\Seg}\from\N^{\rel}\pr Z\to X$ be the Segalification.
Since $\Delta^{l\pr Z}$ is Segal, there is a unique map $f\from X\to\Delta^{l\pr Z}$ rendering the diagram
\[\begin{tikzcd}
	{\N^\rel(Z)} & {\Delta^{l(Z)}} \\
	X
	\arrow[from=1-1, to=1-2]
	\arrow["{\eta_{\Seg}}"', from=1-1, to=2-1]
	\arrow["f"', dashed, from=2-1, to=1-2]
\end{tikzcd}\]
commutative.
By Lemma \ref{lem:MG_Z}, the horizontal map is a complete Segalification.
Thus, it suffices to show that $f$ is an equivalence.

Since $f$ becomes an equivalence after completion and completion does not change mapping animae, $f$ is fully faithful.
It therefore remains to show that $f$ is an equivalence in degree $0$.
The unit $\eta_{\Seg}$ is an equivalence in degree $0$.
Indeed, Segalification is the localization at the strongly saturated class generated by the spine inclusions $\operatorname{Sp}^{n}\to\Delta^{n}$, and evaluation at $[0]$ preserves colimits and sends every spine inclusion to an equivalence.
The horizontal map is also an equivalence in degree $0$, so $f$ is an equivalence in degree $0$.
Hence, $f$ is an equivalence.
\end{proof}

\begin{prop}
\label{prop:Z_Z}For every $Z\in\Zig$, the map $\gamma_{Z}\from\N^{\rel}\pr Z\to\cat Z_{Z}$ exhibits $\cat Z_{Z}$ as the Segalification of $\N^{\rel}\pr Z$.
\end{prop}

\begin{proof}
Given $Z_{1},\dots,Z_{n}\in\Zig$, we have a map $\abs{\Zig_{Z_{1}\cdots Z_{n}}\pr Z}\to\Delta^{l\pr Z}_{n}$, given by the postcomposition with the localization map $Z\to l\pr Z$ and the precomposition with the map $[n]\to l\pr{Z_{1}\cdots Z_{n}}$ that picks up the endpoints of $l\pr{Z_{i}}$.
These maps determine a natural transformation depicted as 
\[\begin{tikzcd}[column sep=5em]
	{(\int\Zig_\partial^{\bullet})^\op} && {\Zig^\op} && {\An.} \\
	{\Del^\op}
	\arrow[""{name=0, anchor=center, inner sep=0}, "{(Z_1,\dots,Z_n)\mapsto Z_1\cdots Z_n}", from=1-1, to=1-3]
	\arrow["\pi"', from=1-1, to=2-1]
	\arrow["{|\Zig_\bullet(Z)|}", from=1-3, to=1-5]
	\arrow[""{name=1, anchor=center, inner sep=0}, "{\Delta^{l(Z)}}"', from=2-1, to=1-5]
	\arrow[between={0.2}{0.8}, Rightarrow, nfold, from=0, to=1]
\end{tikzcd}\]
The universal property of left Kan extensions thus gives a map $q\from\cat Z_{Z}\to\Delta^{l\pr Z}$.
We will show that $q$ is an equivalence.
Since the composite 
\[
\N^{\rel}\pr Z\xrightarrow{\gamma_{Z}}\cat Z_{Z}\xrightarrow{q}\Delta^{l\pr Z}
\]
is a complete Segalification (Lemma \ref{lem:MG_Z}), this and Lemma \ref{lem:Segalification_NrelZ} will prove the desired claim.

Since $q$ is a map of Segal animae (Proposition \ref{prop:segal}), it suffices to show that $q$ induces an equivalence on the $0$th animae and the mapping animae.
The claim on the $0$th animae is clear.
For the claim on mapping animae, let $a,b\in Z$.
We must show that 
\[
\cat Z_{Z}\pr{a,b}\simeq\begin{cases}
\emptyset & \text{if }q_{Z}\pr a>q_{Z}\pr b,\\
\ast & \text{if }q_{Z}\pr a\leq q_{Z}\pr b.
\end{cases}
\]
By Proposition \ref{prop:hammock}, we have
\begin{align}
\cat Z_{Z}\pr{a,b} & \simeq\colim_{X\in\Zig^{\op}_{\partial}}\abs{\Zig_{X}\pr Z_{a,b}}.\label{eq:Z}
\end{align}
If $q_{Z}\pr a>q_{Z}\pr b$, then this colimit is empty, and we are done.
So assume that $q_{Z}\pr a\leq q_{Z}\pr b$.
Let $a_{\min}=\min Z_{q_{Z}\pr a}$ and $b_{\max}=\max Z_{q_{Z}\pr b}$, where the minimum and maximum refer to the natural ordering of integers.
There are weak equivalences $a\to a_{\min}$ and $b_{\max}\to b$, so Corollary \ref{cor:ham_equivalence} gives an equivalence
\[
\cat Z_{Z}\pr{a,b}\simeq\cat Z_{Z}\pr{a_{\min},b_{\max}}.
\]
Replacing $a$ and $b$ by $a_{\min}$ and $b_{\max}$, respectively, we may therefore assume that $a$ is minimal and $b$ is maximal in their fibers under $q_{Z}\from Z\to l\pr Z$.

The functor $\Zig_{\bullet}\pr Z_{a,b}\from\Zig^{\op}_{\partial}\to\Cat$ classifies a cartesian fibration $\cat X \to \Zig_{\partial}$, and the right-hand side of (\ref{eq:Z}) is the realization of $\cat X$. It thus remains to verify the following claim:

\textbf{Claim:} The realization $\abs{\cat X}$ is contractible.

Let us lay out the proof strategy for this claim. The objects of $\cat X$ are pairs $\pr{X,f}$, where $X\in\Zig_{\partial}$ and $f\in\Zig_X\pr Z_{a,b}$. A morphism $\pr{X,f}\to\pr{Y,g}$ is a pair $\pr{u,\alpha}$, where $u\from X\to Y$ is a morphism in $\Zig_{\partial}$ and $\alpha\from f\to gu$ is a pointwise weak equivalence. Since $l\pr Z$ is an ordinal, the pair $\pr{X,f}$ may be regarded as a zig-zag presentation of the unique morphism $q_Z\pr a\to q_Z\pr b$ in $l\pr Z$. We will put these presentations into a normal form by successively removing three kinds of inessential data: nonmonotone choices within the fibers of $q_Z$, right-pointing arrows mapped to identities, and redundant subdivisions of strings of weak equivalences. This results in a sequence of subcategories $\cat X_2 \subseteq \cat X_1 \subseteq \cat X_0 \subseteq \cat X$ that induce equivalences on realizations, after which we show directly that $\cat X_2$ is contractible.

\emph{Step 1: Removing nonmonotone choices.}
Let $\cat X_0\subset\cat X$ be the full subcategory spanned by the pairs $\pr{X,f}$ for which $f$ preserves the natural ordering of integers.
The projection $\cat X_0\to\Zig_{\partial}$ is a cartesian subfibration of $\cat X\to\Zig_{\partial}$, since precomposition by a morphism of $\Zig_{\partial}$ preserves the natural ordering.
We claim that the inclusion $\cat X_0\to\cat X$ is a weak homotopy equivalence.
By Corollary \ref{cor:fiberwise_pullback}, it suffices to prove that for every $X\in\Zig_{\partial}$ the map $\abs{(\cat X_0)_X}\to\abs{\cat X_X}$ is an equivalence.

Note that a pointwise weak equivalence between two functors $X\to Z$ does not change their composite with $q_Z$.
Thus, the fiber $\cat X_X$ decomposes according to the maps $p\from X\to l\pr Z$ which arise in this way. Fixing one such $p$, we denote the resulting components by $\cat X_{X,p}$ and $(\cat X_0)_{X,p}$.

We next construct an initial object of $\cat X_{X,p}$ and show it is contained in $\cat X_0$. For every $i\in l\pr Z$, write $X_i=p^{-1}\pr i$; if $X_i$ is nonempty, let $x_i\in X_i$ be the maximal integer which receives a morphism from $\min X_i$.
Define $f_0\from X\to Z$ by
\[
f_0\pr x=\begin{cases}
\min Z_{p\pr x} & \text{if }x\leq x_{p\pr x},\\
\max Z_{p\pr x} & \text{if }x>x_{p\pr x}.
\end{cases}
\]
This formula defines an object $f_0\in(\cat X_0)_{X,p}$.
The key points are that within each $X_i$ its only change of value occurs across a left-pointing arrow, which is sent to the weak morphism $\max Z_i\to\min Z_i$, while an arrow on which $p$ increases from $i$ to $j$ is sent along the unique directed path from $\max Z_i$ to $\min Z_j$.
The fact that $p$ admits a lift ensures that the formula takes the source and target of the latter arrow to the indicated extrema, including in the degenerate cases.
The required endpoints follow from our assumptions on $a$ and $b$, and it is clear from the formula that $q_Zf_0=p$ and that $f_0$ preserves the natural ordering.

We claim that $f_0$ is initial in $\cat X_{X,p}$.
Indeed, let $g\in\cat X_{X,p}$ and $x\in X$.
Suppose first that $x\leq x_{p\pr x}$.
The vertex $\min X_{p\pr x}$ is either the initial boundary or the target of an arrow on which $p$ increases, and is therefore sent by $g$ to $\min Z_{p\pr x}$.
Moreover, $x$ is reached from this vertex by a right-pointing path, while every nonidentity morphism in $Z_{p\pr x}$ points to the left; hence $g\pr x=\min Z_{p\pr x}$.
If $x>x_{p\pr x}$, there is a unique weak morphism $\max Z_{p\pr x}\to g\pr x$.
Thus, in either case there is a unique weak morphism $f_0\pr x\to g\pr x$.
These morphisms determine a unique pointwise weak equivalence $f_0\to g$; naturality follows from the fact that $Z$ has at most one morphism between any two objects.
We conclude that $f_0$ is initial both in $\cat X_{X,p}$ and in $(\cat X_0)_{X,p}$, so $\abs{(\cat X_0)_{X,p}}\to\abs{\cat X_{X,p}}$ is an equivalence.
Since this holds for every $p$, it follows that $\abs{(\cat X_0)_X}\to\abs{\cat X_X}$ is an equivalence, and hence so is $\abs{\cat X_0}\to\abs{\cat X}$.

\emph{Step 2: Removing identity subdivisions.}
Let $\cat X_1\subset\cat X_0$ be the full subcategory spanned by the pairs $\pr{X,f}$ for which every right-pointing arrow of $X$ strictly increases $q_Zf$. We show that the inclusion admits a left adjoint, so that $\abs{\cat X_1}\to\abs{\cat X_0}$ is an equivalence. It suffices to construct this adjoint pointwise for each $\pr{X,f}\in\cat X_0$. Write again $X_i$ for the fibers of $q_Zf$. Since $f$ preserves the natural ordering and all nonidentity morphisms in a fiber of $q_Z$ point to the left, every right-pointing arrow in $X_i$ is sent to an identity. Let $X'$ be the zig-zag obtained by collapsing all these arrows, and write $\pi\from X\to X'$ for the quotient map. There is a unique relative functor $f'\from X' \to Z$ such that $f=f'\pi$. An analogous reasoning shows that $\pi$ exhibits $(X',f')$ as a left-adjoint object to $(X,f)$ under the inclusion $\cat X_1 \hookrightarrow \cat X_0$.

\emph{Step 3: Removing weak subdivisions.}
Let $\cat X_2\subset\cat X_1$ be the full subcategory spanned by the pairs $\pr{X,f}$ for which every nonempty fiber of $q_Zf$ has at most two objects. We claim that the inclusion admits a right adjoint, so that $\abs{\cat X_2}\to\abs{\cat X_1}$ is an equivalence. We may again construct this adjoint pointwise for each $\pr{X,f}\in\cat X_1$. Each nonempty fiber of $q_Zf$ is a left-pointing chain, and so internal vertices only record a subdivision of the weak morphism from its maximum to its minimum.
Restricting $f$ to the full subcategory of $X$ spanned by the minima and maxima of these fibers gives an object $(X'',f'')$ of $\cat X_2$. The inclusion exhibits it as a right adjoint object to $(X,f)$ under the inclusion. To see this, it suffices to observe that any morphism in $\cat X_1$ preserves the minimum and maximum of every nonempty fiber. Indeed, for an internal fiber they are characterized as the target of the entering morphism and the source of the exiting morphism, respectively, while for the first and last fibers this follows from preservation of the boundary.

\emph{Step 4: Contracting the normal forms.}
Let $Z''\subset Z$ be the full subcategory spanned by the minima and maxima of the fibers $Z_i$ with $q_Z\pr a\leq i\leq q_Z\pr b$, and let $\iota\from Z''\hookrightarrow Z$ denote the inclusion.
Then $\pr{Z'',\iota}$ is an object of $\cat X_2$.
It is terminal: For every $\pr{X,f}\in\cat X_2$, each vertex of $X$ is the minimum or maximum of a nonempty fiber of $q_Zf$, so the characterization of these extrema in Step 3 shows that $f$ factors through $Z''$ and that this factorization defines the unique morphism $\pr{X,f}\to\pr{Z'',\iota}$.
Hence, $\abs{\cat X_2}$ is contractible.
Combining the four steps, we conclude that $\abs{\cat X}$ is contractible.
\end{proof}

The following construction packages the previous proposition into a Yoneda argument.
It is designed so that, for every $Z\in\Zig$, the presheaf $X\mapsto\abs{\Zig_{X}\pr Z}$ is representable.

\begin{construction}
\label{const:overlineZig}
We define an $\infty$-category $\overline{\Zig}$ as follows.
First, observe that the $\infty$-category $\RelCat_{\infty}$ is cartesian closed: The internal mapping object from $\pr{\cat C,\cat W}$ to $\pr{\cat D,\cat V}$ is given by
\[
\pr{\Fun_{\rel}\pr{\pr{\cat C,\cat W},\pr{\cat D,\cat V}},\Fun_{\rel}\pr{\pr{\cat C,\cat W},\pr{\cat D,\cat V}}\times_{\Fun\pr{\cat C^{\simeq},\cat D}}\Fun\pr{\cat C^{\simeq},\cat V}}.
\]
In particular, $\RelCat_{\infty}$ is enriched over itself.
Changing enrichments along the product-preserving functor
\[
\RelCat_{\infty}\to\An,\,\pr{\cat C,\cat W}\mapsto\abs{\cat W},
\]
we obtain a new $\infty$-category, denoted by $\overline{\RelCat_{\infty}}$.
We write $\overline{\Zig}\subset\overline{\RelCat_{\infty}}$ for the full subcategory spanned by the objects in $\Zig$.
Note that $\Map_{\overline{\Zig}}\pr{Z,Z'}\simeq\abs{\Zig_{Z}\pr{Z'}}$ by definition.
\end{construction}

\begin{proof}
[Proof of Proposition \ref{prop:segalification}]
Let $\Seg\pr{\An}\subset\PSh\pr{\Del}$ denote the full subcategory spanned by Segal animae, let $L_{\Seg}\from\PSh\pr{\Del}\to\Seg\pr{\An}$ denote the Segalification functor, and let $f\from\int\Zig^{\bullet}_{\partial}\to\overline{\Zig}$ denote the functor $\pr{Z_{1},\dots,Z_{n}}\mapsto Z_{1}\cdots Z_{n}$.
We consider the composite natural transformation
\[\begin{tikzcd}
	{\PSh(\overline{\Zig})} & {\PSh(\int\Zig_\partial ^\bullet )} && {\PSh(\Del)} & {\Seg(\An)}
	\arrow["{f^*}", from=1-1, to=1-2]
	\arrow[""{name=0, anchor=center, inner sep=0}, "{\iota^*}", curve={height=-12pt}, from=1-2, to=1-4]
	\arrow[""{name=1, anchor=center, inner sep=0}, "{\pi_!}"', curve={height=12pt}, from=1-2, to=1-4]
	\arrow["{L_{\Seg}}", from=1-4, to=1-5]
	\arrow["{\iota^*\eta}", between={0.2}{0.8}, Rightarrow, nfold, from=0, to=1]
\end{tikzcd}\]
where $\eta$ denotes the unit of the adjunction $\pi_{!}\dashv\pi^{*}$.
For $Z\in\Zig$, let $h_{Z}=\Map_{\overline{\Zig}}\pr{-,Z}$ be the corresponding representable presheaf.
There are canonical identifications
\[
\iota^{*}f^{*}h_{Z}\simeq\N^{\rel}\pr Z,
\qquad
\pi_{!}f^{*}h_{Z}\simeq\cat Z_{Z},
\]
under which $\iota^{*}\eta$ is identified with $\gamma_{Z}$.
Thus, Proposition \ref{prop:Z_Z} shows that the displayed natural transformation evaluates to an equivalence at every representable.
Both composites in the displayed diagram preserve colimits, and the representables generate $\PSh\pr{\overline{\Zig}}$ under colimits, so it evaluates to an equivalence at every presheaf.

Now consider the presheaf $P_{\cat C,\cat W}=\abs{\Zig_{\bullet}\pr{\cat C,\cat W}}\in\PSh\pr{\overline{\Zig}}$, where the presheaf functoriality comes from the mapping-anima functor of $\overline{\RelCat_{\infty}}$. By construction, there are canonical identifications
\[
\iota^{*}f^{*}P_{\cat C,\cat W}\simeq\N^{\rel}\pr{\cat C,\cat W},
\qquad
\pi_{!}f^{*}P_{\cat C,\cat W}\simeq\cat Z_{\cat C,\cat W},
\]
under which $\iota^{*}\eta$ is identified with $\gamma_{\cat C,\cat W}$.
Since $\cat Z_{\cat C,\cat W}$ is Segal by Proposition \ref{prop:segal}, the claim follows.
\end{proof}

\subsection{\label{subsec:summary}Proof of Theorem \ref{thm:Z_Segal}}

Summarizing our discussion so far, we obtain Theorem \ref{thm:Z_Segal}: Part (1) is the content of Proposition \ref{prop:segal}; part (2) is the content of Proposition \ref{prop:hammock}; and part (3) is the content of Proposition \ref{prop:segalification}.

\section{\label{sec:fraction}Calculus of fractions}

As we saw in Theorem \ref{thm:Z_Segal}, given a relative $\infty$-category $\pr{\cat C,\cat W}$, morphisms in the localization $\cat C[\cat W^{-1}]$ can be represented by a zig-zag of morphisms in $\cat C$.
The theory of \textit{calculus of fractions}, initiated by Gabriel--Zisman \cite{GZ67} and extended by Dwyer--Kan \cite{DK80_2}, provides a systematic method to reduce such zig-zags to shorter ones.
The fraction calculations and their proofs are essentially already contained in Dwyer--Kan's work; our purpose is to formulate them for relative $\infty$-categories and relate them to the model-independent constructions of the preceding section.
In particular, we give a sufficient condition for the mapping animae $\Hamm_{\cat C,\cat W}\pr{x,y}$ of $\cat C[\cat W^{-1}]$ to admit much simpler descriptions (Theorem \ref{thm:fraction}).
Together with the preceding section, this recovers Mazel-Gee's Segal criterion for the relative Rezk nerve \cite[Theorem 6.1]{MG18} (Corollary \ref{cor:fraction}).

We start with the definition of calculus of fractions.

\begin{defn}
\label{def:calculus}Let $\pr{\cat C,\cat W}$ be a relative $\infty$-category. 
\begin{enumerate}
\item We say that $\pr{\cat C,\cat W}$ \textit{supports left fractions} if, for every pair of integers $i,j\geq0$, and every pair of objects $x,y\in\cat C$, the map
\[
\Zig_{[i+j,-1]}\pr{\cat C,\cat W}_{x,y}\to\Zig_{[i,-1,j,-1]}\pr{\cat C,\cat W}_{x,y}
\]
is a weak homotopy equivalence.
\item We say that $\pr{\cat C,\cat W}$ \textit{supports right fractions} if, for every pair of integers $i,j\geq0$, and every pair of objects $x,y\in\cat C$, the map
\[
\Zig_{[-1,i+j]}\pr{\cat C,\cat W}_{x,y}\to\Zig_{[-1,i,-1,j]}\pr{\cat C,\cat W}_{x,y}
\]
is a weak homotopy equivalence.
\item We say that $\pr{\cat C,\cat W}$ \textit{supports two-sided fractions} if, for every pair of integers $i,j>0$, and every pair of objects $x,y\in\cat C$, the map
\[
\Zig_{[-1,i+j,-1]}\pr{\cat C,\cat W}_{x,y}\to\Zig_{[-1,i,-1,j,-1]}\pr{\cat C,\cat W}_{x,y}
\]
is a weak homotopy equivalence.
\end{enumerate}
\end{defn}

\begin{rem}
For a relative $1$-category $\pr{\cat C,\cat W}$, the categories of zig-zags and the comparison maps appearing in Definition \ref{def:calculus} agree with the corresponding categories of words and comparison maps used by Dwyer--Kan in \cite{DK80_2}.
In their terminology, $\pr{\cat C,\cat W}$ admits a homotopy calculus of left, right, or two-sided fractions if and only if both $\pr{\cat C,\cat W}$ and $\pr{\cat W,\cat W}$ support fractions of the corresponding type.
Thus, our condition is weaker, and this is all that we need.

If $\cat W$ has the two-out-of-three property, then $\Zig_Z\pr{\cat W,\cat W}\subset\Zig_Z\pr{\cat C,\cat W}$ is a union of components, and each comparison map in Definition \ref{def:calculus} carries the union in its source precisely onto the union in its target.
Consequently, if $\pr{\cat C,\cat W}$ satisfies two-out-of-three and supports left, right, or two-sided fractions in our sense, then it admits a homotopy calculus of fractions of the corresponding type in the sense of Dwyer--Kan.
\end{rem}

\begin{rem}
\label{rem:left_right_double}We will see below that a relative $\infty$-category supporting left or right fractions supports two-sided fractions (Lemma \ref{lem:left_right_double}).
\end{rem}

Here is the main result of this section.

\begin{thm}
\label{thm:fraction}Let $\pr{\cat C,\cat W}$ be a relative $\infty$-category, and let $x,y\in\cat C$ be objects.

\begin{enumerate}
\item If $\pr{\cat C,\cat W}$ supports left fractions, then the map
\[
\abs{\Zig_{[1,-1]}\pr{\cat C,\cat W}_{x,y}}\to\Hamm_{\cat C,\cat W}\pr{x,y}
\]
is an equivalence.
\item If $\pr{\cat C,\cat W}$ supports right fractions, then the map
\[
\abs{\Zig_{[-1,1]}\pr{\cat C,\cat W}_{x,y}}\to\Hamm_{\cat C,\cat W}\pr{x,y}
\]
is an equivalence.
\item If $\pr{\cat C,\cat W}$ supports two-sided fractions, then the map
\[
\abs{\Zig_{[-1,1,-1]}\pr{\cat C,\cat W}_{x,y}}\to\Hamm_{\cat C,\cat W}\pr{x,y}
\]
is an equivalence.
\end{enumerate}
\end{thm}

Combining Theorem \ref{thm:fraction} with Theorem \ref{thm:Z_Segal}, we can also give the following sufficient condition for the relative Rezk nerve to be (complete) Segal.
This implication already appears in \cite[Theorem 6.1]{MG18}, but we will record a proof for the reader's convenience.

\begin{cor}
\label{cor:fraction}For every relative $\infty$-category $\pr{\cat C,\cat W}$ supporting left, right, or two-sided fractions, $\N^{\rel}\pr{\cat C,\cat W}$ is a Segal anima.
If, in addition, $\cat W$ has the two-out-of-three property, then it is complete if and only if $\pr{\cat C,\cat W}$ is saturated, i.e., maps in $\cat W$ are exactly the maps whose images in $\cat C[\cat W^{-1}]$ are equivalences.
\end{cor}

\begin{rem}
\label{rem:fraction_naturality}Let $\pr{\cat C,\cat W}$ be a relative $\infty$-category supporting left fractions.
Then we have an equivalence
\[
\Map_{\cat C[\cat W^{-1}]}\pr{x,y}\simeq\colim_{y\to y'\in\cat W_{y/}}\Map_{\cat C}\pr{x,y'}
\]
which is natural in $x\in\cat C$ (not just in $x\in\cat W$).
Indeed, for fixed $y$, the target-evaluation functor $\cat W_{y/}\to\cat C$ is a putative left calculus of fractions at $y$ in the sense of Cisinski, so by \cite[Theorem 7.2.8]{HCHA} it suffices to show that the presheaf
\[
x\mapsto\colim_{y\to y'\in\cat W_{y/}}\Map_{\cat C}\pr{x,y'}
\]
sends morphisms in $\cat W$ to equivalences.
This follows from Corollary \ref{cor:ham_equivalence} and the identification
\[
\colim_{y\to y'\in\cat W_{y/}}\Map_{\cat C}\pr{x,y'}\simeq\abs{\cat C_{x/}\times_{\cat C}\cat W_{y/}}\simeq\abs{\Zig_{[1,-1]}\pr{\cat C,\cat W}_{x,y}}.
\]
\end{rem}

\begin{rem}
Mazel-Gee proved part (3) of Theorem \ref{thm:fraction} in \cite[Theorem 4.4]{MG18}, and Corollary \ref{cor:fraction} in \cite[Theorem 6.1]{MG18}.
Our proof of these results has some overlap with his, but is shorter.
We also note that his result is somewhat unsatisfactory, as he did not prove, for a general relative $\infty$-category, that $\Hamm_{\cat C,\cat W}\pr{x,y}$ is equivalent to $\cat C[\cat W^{-1}]\pr{x,y}$ (which we proved in Corollary \ref{cor:hammock_anima_is_anima_of_loc}); see Section \ref{sec:MGcomparison}.
When $\pr{\cat C,\cat W}$ supports two-sided fractions, this equivalence can be deduced from Theorem \ref{thm:univ_Rezk} together with the identification of the mapping animae of $\N^{\rel}\pr{\cat C,\cat W}$ that is implicit in the proof of \cite[Theorem 6.1]{MG18}, but it is not stated there.
\end{rem}

\begin{rem}
\cite[Theorem 4.11]{LM15}\label{rem:saturated} Given a relative $\infty$-category $\pr{\cat C,\cat W}$ satisfying the two-out-of-three property and such that $\N^{\rel}\pr{\cat C,\cat W}$ is Segal, the following conditions are equivalent:

\begin{enumerate}
\item $\N^{\rel}\pr{\cat C,\cat W}$ is a complete Segal anima.
\item $\pr{\cat C,\cat W}$ is saturated.
\end{enumerate}

The implication (1)$\implies$(2) follows from the observation that, by the completeness condition, every morphism $f$ in $\cat C$ inverted in $\cat C[\cat W^{-1}]$ must lie in the same component of $\N^{\rel}\pr{\cat C,\cat W}_{1}$ as an identity morphism.
For (2)$\implies$(1), note that the map $\N^{\rel}\pr{\cat C,\cat W}\to\N\pr{\cat C[\cat W^{-1}]}$ is a complete Segalification by Mazel-Gee's theorem (Theorem \ref{thm:univ_Rezk}).
Thus, the full subanima $\N^{\rel}\pr{\cat C,\cat W}_{\mathrm{eq}}\subset\N^{\rel}\pr{\cat C,\cat W}_{1}$ of equivalences is exactly the preimage of $\Map\pr{[1],\cat C[\cat W^{-1}]^{\simeq}}$. It follows by saturation and the definition of $\N^{\rel}\pr{\cat C,\cat W}$ that this is $\abs{\Fun\pr{[1],\cat W}}$.
Since $\abs{\cat W}\to\abs{\Fun\pr{[1],\cat W}}$ is an equivalence, this implies (1).
\end{rem}

The remainder of this section is devoted to the proof of Theorem \ref{thm:fraction} and Corollary \ref{cor:fraction}.
We start with the proof of Theorem \ref{thm:fraction}.
The proof of this theorem is essentially contained in \cite[Proposition 6.2]{DK80_2}, but since our terminology and notation differ from those of Dwyer--Kan, we record the argument in our setting.

\begin{notation}
Let $\pr{\cat C,\cat W}$ be a relative $\infty$-category, and let $x,y\in\cat C$ be objects, and let $Z\in\Zig$.
We set
\[
\ps{\Zig_{Z}\pr{\cat C,\cat W}}x=\{x\}\times_{\cat W}\Zig_{Z}\pr{\cat C,\cat W},
\]
where the map $\Zig_{Z}\pr{\cat C,\cat W}\to\cat W$ is the evaluation at the left endpoint.
We think of $\ps{\Zig_{Z}\pr{\cat C,\cat W}}x$ as lying over $\cat W$ via the evaluation at the right endpoint.
Note that, if $Z\neq[0]$, the map $\ps{\Zig_{Z}\pr{\cat C,\cat W}}x\to\cat W$ is a cocartesian fibration if the rightmost arrow of $Z$ points to the right, and is a cartesian fibration if it points to the left.

Likewise, we set
\[
\Zig_{Z}\pr{\cat C,\cat W}_{y}=\Zig_{Z}\pr{\cat C,\cat W}\times_{\cat W}\{y\},
\]
where the map $\Zig_{Z}\pr{\cat C,\cat W}\to\cat W$ is the evaluation at the right endpoint.
We think of $\Zig_{Z}\pr{\cat C,\cat W}_{y}$ as lying over $\cat W$ via the evaluation at the left endpoint.
\end{notation}

\begin{lem}
\label{lem:6.2_1}Let $\pr{\cat C,\cat W}$ be a relative $\infty$-category, let $x,y\in\cat C$ be objects, and let $Z\in\Zig$.

\begin{enumerate}
\item If $\pr{\cat C,\cat W}$ supports left fractions, then the map
\[
\Zig_{[l\pr Z,-1]}\pr{\cat C,\cat W}_{x,y}\to\Zig_{[Z,-1]}\pr{\cat C,\cat W}_{x,y}
\]
is a weak homotopy equivalence.
\item If $\pr{\cat C,\cat W}$ supports right fractions, then the map
\[
\Zig_{[-1,l\pr Z]}\pr{\cat C,\cat W}_{x,y}\to\Zig_{[-1,Z]}\pr{\cat C,\cat W}_{x,y}
\]
is a weak homotopy equivalence.
\item If $\pr{\cat C,\cat W}$ supports two-sided fractions, then the map
\[
\Zig_{[-1,l\pr Z,-1]}\pr{\cat C,\cat W}_{x,y}\to\Zig_{[-1,Z,-1]}\pr{\cat C,\cat W}_{x,y}
\]
is a weak homotopy equivalence.
\end{enumerate}
\end{lem}

\begin{proof}
We prove (2).
Composing consecutive left-pointing arrows gives a homotopy inverse, after realization, to the operation of inserting identity arrows into a consecutive string of left-pointing arrows.
This remains true after concatenating arbitrary zig-zags on either side, since the functors and natural transformations involved preserve the endpoints.
Thus, we may assume that $Z$ contains no consecutive left-pointing arrows or a left-pointing arrow at the left end.
Write $\theta$ for the map in question, and write $Z=[S,T]$. We prove the claim by induction on $n=\abs T$.

If $n=0$, the claim is trivial.
If $n=1$, the claim is precisely the right-fraction hypothesis.
For the inductive step, suppose $n\geq2$ and that the claim has been proved for $n-1$. Write $Z=[Z_{1},-1,Z_{2}]$, and factor $\theta$ as
\begin{align*}
\Zig_{[-1,l\pr Z]}\pr{\cat C,\cat W}_{x,y} & \simeq\Zig_{[-1,l\pr{Z_{1}},l\pr{Z_{2}}]}\pr{\cat C,\cat W}_{x,y}\\
 & \simeq\ps{\Zig_{[-1,l\pr{Z_{1}}]}\pr{\cat C,\cat W}}x\times_{\cat W}\Zig_{l\pr{Z_{2}}}\pr{\cat C,\cat W}_{y}\\
 & \xrightarrow{f}\ps{\Zig_{[-1,l\pr{Z_{1}},-1]}\pr{\cat C,\cat W}}x\times_{\cat W}\Zig_{l\pr{Z_{2}}}\pr{\cat C,\cat W}_{y}\\
 & \simeq\ps{\Zig_{[-1,l\pr{Z_{1}}]}\pr{\cat C,\cat W}}x\times_{\cat W}\Zig_{[-1,l\pr{Z_{2}}]}\pr{\cat C,\cat W}_{y}\\
 & \xrightarrow{g}\ps{\Zig_{[-1,Z_{1}]}\pr{\cat C,\cat W}}x\times_{\cat W}\Zig_{[-1,Z_{2}]}\pr{\cat C,\cat W}_{y}.
\end{align*}
The map $f$ is an equivalence by the right-fraction hypothesis.
For $g$, all four categories in the displayed pullbacks are cocartesian fibrations over $\cat W$; here we use the fact that $Z_{1}$ has the form $[\ldots,1]$, which is ensured by the reduction in the first paragraph.
The induction hypothesis shows that the two maps defining $g$ are equivalences on every fiber after realization.
Thus, Corollary \ref{cor:fiberwise_pullback} shows that $g$ is a weak homotopy equivalence.
It follows that the composite is also an equivalence, completing the induction.

The proof of (1) is dual.
For (3), the same argument applies after making the reduction of the first paragraph at both ends.
Consequently, when only one left-pointing arrow remains, the right-pointing strings on both sides are nonempty, as required by the definition of two-sided fractions.
\end{proof}

\begin{notation}
We write $\Del_{\partial}\subset\Zig_{\partial}$ for the full subcategory spanned by the objects $[k]$, where $k\geq0$.
\end{notation}

\begin{lem}
\label{lem:6.2_2}The functor $l\from\Zig_{\partial}\to\Del_{\partial}$ is a left adjoint, and hence is initial.
\end{lem}

\begin{proof}
The right adjoint is the inclusion: Every relative functor $Z\to[k]$ sends the left-pointing arrows of $Z$ to identities, and hence factors uniquely through $Z\to l\pr Z$.
\end{proof}

We will also use the following elementary observation about changing indexing categories of colimits.

\begin{lem}
\label{lem:colim}Let
\[
I\xrightarrow{f}J\xrightarrow{X}\cat E
\]
be functors of $\infty$-categories.
Suppose there are functors $g\from J\to I$ and natural transformations $\eta\from\id_{I}\to gf$ and $\eta'\from\id_{J}\to fg$ satisfying $Xf\eta\simeq X\eta'f$.
Then the canonical map
\[
\colim_{i\in I}Xf\pr i\to\colim_{j\in J}X\pr j
\]
is an equivalence, provided that the colimits exist.
\end{lem}

\begin{proof}
The inverse equivalence is induced by the composites
\[
X\pr j\xrightarrow{X\eta'_{j}}Xfg\pr j\xrightarrow{\iota_{g\pr j}}\colim_{i\in I}Xf\pr i,
\]
where $\iota$ denotes the colimit cone for $\colim_{i\in I}Xf\pr i$.
Postcomposition with the canonical map in the statement gives the identity by naturality of the colimit cone for $X$. In the other order, the relation $Xf\eta\simeq X\eta'f$ identifies the composite on each $Xf\pr i$ with the structure map $\iota_i$ of the colimit cone.
\end{proof}

\begin{proof}
[Proof of Theorem \ref{thm:fraction}] We will prove (2); the other cases can be proved similarly. 

Consider the following diagram: 
\[\begin{tikzcd}
	{\Zig_{[-1,1]}(\cat{C},\cat{W})_{x,y}} & {\Zig_{[-1,-1,1]}(\cat{C},\cat{W})_{x,y}} \\
	{\colim_{Z\in \Zig_{\partial}^{\op}}(\Zig_{Z}(\cat{C},\cat{W})_{x,y})} & {\colim_{Z\in \Zig_{\partial}^{\op}}(\Zig_{[-1,Z]}(\cat{C},\cat{W})_{x,y}).}
	\arrow[""{name=0, anchor=center, inner sep=0}, "{s_1}", curve={height=-12pt}, from=1-1, to=1-2]
	\arrow[""{name=1, anchor=center, inner sep=0}, "{s_0}"', curve={height=12pt}, from=1-1, to=1-2]
	\arrow["\phi"', from=1-1, to=2-1]
	\arrow["\psi", from=1-2, to=2-2]
	\arrow["{\overline{s}_0}"', from=2-1, to=2-2]
	\arrow["\alpha", between={0.2}{0.8}, Rightarrow, from=0, to=1]
\end{tikzcd}\]
Here, the maps $s_{0},s_{1}$ at the top are induced by the surjections $[-1,-1,1]\to[-1,1]$ that repeat $0$ and $1$ twice, respectively, and the map $\overline{s}_{0}$ at the bottom is defined similarly.
The maps $\phi$ and $\psi$ are components of the relevant colimit diagrams, and $\alpha$ is the natural transformation which we can depict as 
\[\begin{tikzcd}
	a & b & b & c \\
	a & a & b & {c.}
	\arrow[equal, from=1-1, to=2-1]
	\arrow[from=1-2, to=1-1]
	\arrow[from=1-2, to=2-2]
	\arrow[equal, from=1-3, to=1-2]
	\arrow[from=1-3, to=1-4]
	\arrow[equal, from=1-3, to=2-3]
	\arrow[equal, from=1-4, to=2-4]
	\arrow[equal, from=2-2, to=2-1]
	\arrow[from=2-3, to=2-2]
	\arrow[from=2-3, to=2-4]
\end{tikzcd}\]
Our goal is to show that $\abs{\phi}$ is a homotopy equivalence.
Since $\abs{s_{0}}\simeq\abs{s_{1}}$, we have $\abs{\psi\circ s_{1}}\simeq\abs{\psi\circ s_{0}}\simeq\abs{\overline{s}_{0}\circ\phi}$.
Thus, it suffices to show that $\abs{\psi\circ s_{1}}$ and $\abs{\overline{s}_{0}}$ are equivalences.

We start with $\psi\circ s_{1}$.
Consider the diagram
\[\begin{tikzcd}
	{\Zig_{[-1,1]}(\cat{C},\cat{W})_{x,y}} & {\colim_{[k]\in\Del_{\partial}^{\op}}(\Zig_{[-1,k]}(\cat{C},\cat{W})_{x,y}),} \\
	{\colim_{Z\in\Zig_{\partial}^{\op}}(\Zig_{[-1,Z]}(\cat{C},\cat{W})_{x,y})} & {\colim_{Z\in\Zig_{\partial}^{\op}}(\Zig_{[-1,l(Z)]}(\cat{C},\cat{W})_{x,y})}
	\arrow["{i_0}", from=1-1, to=1-2]
	\arrow["{\psi\circ s_1}"', from=1-1, to=2-1]
	\arrow["{i_1}", from=1-1, to=2-2]
	\arrow["v"', from=2-2, to=1-2]
	\arrow["u", from=2-2, to=2-1]
\end{tikzcd}\]
where $i_{0},i_{1},\psi\circ s_{1}$ are components of the relevant colimit diagrams, $u$ is induced by the maps $Z\to l\pr Z$, and $v$ is induced by the functor $l\from\Zig_{\partial}\to\Del_{\partial}$.
Lemmas \ref{lem:6.2_1} and \ref{lem:6.2_2} show that $\abs u$ and $\abs v$ are equivalences.
The map $\abs{i_{0}}$ is also an equivalence, because $[1]\in\Del_{\partial}$ is initial.
It follows from the two-out-of-three property of equivalences that $\abs{\psi\circ s_{1}}$ is an equivalence, as claimed.

Next, we turn to $\overline{s}_{0}$.
Define functors $f\from\Zig_{\partial}^{\op}\to\Zig_{\partial}^{\op}$ and $X\from\Zig^{\op}_{\partial}\to\An$ by $f\pr Z=[-1,Z]$ and $X\pr Z=\abs{\Zig_{Z}\pr{\cat C,\cat W}_{x,y}}$.
The surjections $[-1,Z]\to Z$ that hit $0$ twice determine a natural transformation $\eta\from\id\to f$ on $\Zig_{\partial}^{\op}$.
The map $\abs{\overline{s}_{0}}$ is the map
\[
\colim X\to\colim Xf
\]
induced by $X\eta$.
This map has a retraction
\[
\rho\from\colim Xf\to\colim X,
\]
induced by $f$.
Thus, it suffices to show that $\rho$ is an equivalence.
For this, we use Lemma \ref{lem:colim} (applied to the case $g=\id$).
According to this lemma, we only have to show that $Xf\eta\simeq X\eta f$.
For $Z\in\Zig_{\partial}$, the maps $Xf\eta_{Z}$ and $X\eta_{f\pr Z}$ are the maps
\[
\abs{\Zig_{[-1,Z]}\pr{\cat C,\cat W}_{x,y}}\to\abs{\Zig_{[-1,-1,Z]}\pr{\cat C,\cat W}_{x,y}}
\]
induced by the surjections $[-1,-1,Z]\to[-1,Z]$ that hit either $0$ or $1$ twice.
A variant of the transformation $\alpha$ then shows that $Xf\eta\simeq X\eta f$, as required.
\end{proof}

Before proving Corollary \ref{cor:fraction}, we show that it suffices to consider two-sided fractions.
\begin{lem}
\label{lem:left_right_double}Let $\pr{\cat C,\cat W}$ be a relative $\infty$-category supporting left or right fractions.
Then $\pr{\cat C,\cat W}$ supports two-sided fractions.
\end{lem}

\begin{proof}
We prove the claim for right fractions; the left case is dual.
Let $x,y\in\cat C$. We must show that, for every $i,j>0$, the map
\[
\Zig_{[-1,i+j,-1]}\pr{\cat C,\cat W}_{x,y}\to\Zig_{[-1,i,-1,j,-1]}\pr{\cat C,\cat W}_{x,y}
\]
is a weak homotopy equivalence. 
But this map can be written as
\[
\ps{\Zig_{[-1,i+1]}\pr{\cat C,\cat W}}x\times_{\cat W}\Zig_{[j-1,-1]}\pr{\cat C,\cat W}_{y}\to\ps{\Zig_{[-1,i,-1,1]}\pr{\cat C,\cat W}}x\times_{\cat W}\Zig_{[j-1,-1]}\pr{\cat C,\cat W}_{y}.
\]
If $j>1$, the two left-hand factors are cocartesian fibrations over $\cat W$, while the common right-hand factor is a cartesian fibration.
If $j=1$, all three fibrations are cocartesian.
In either case, the right-fraction hypothesis gives an equivalence on every fiber after realization, so the claim follows from Corollary \ref{cor:fiberwise_pullback}.
\end{proof}

\begin{proof}
[Proof of Corollary \ref{cor:fraction}]
By Lemma \ref{lem:left_right_double}, it suffices to consider a relative $\infty$-category supporting two-sided fractions.
Let $n\geq2$, and consider the following commutative diagram
\[\hspace{-3pt}\begin{tikzcd}[column sep = -5pt]
	& {|\Zig_{[-1,1,-1][-1,n-1,-1]}(\cat{C},\cat{W})|} && {|\Zig_{[-1,n-1,-1]}(\cat{C},\cat{W})|} \\
	{|\Zig_{[n]}(\cat{C},\cat{W})|} && {|\Zig_{[n-1]}(\cat{C},\cat{W})|} \\
	& {|\Zig_{[-1,1,-1]}(\cat{C},\cat{W})|} && {|\cat{W}|.} \\
	{|\Zig_{[1]}(\cat{C},\cat{W})|} && {|\cat{W}|}
	\arrow[from=1-2, to=1-4]
	\arrow[from=1-2, to=3-2]
	\arrow[from=1-4, to=3-4]
	\arrow[from=2-1, to=1-2]
	\arrow[crossing over, from=2-1, to=2-3]
	\arrow[from=2-1, to=4-1]
	\arrow[from=2-3, to=1-4]
	\arrow[from=3-2, to=3-4]
	\arrow["{|\dom|}"{pos=0.7}, crossing over, from=2-3, to=4-3]
	\arrow[from=4-1, to=3-2]
	\arrow["{|\codom|}"', from=4-1, to=4-3]
	\arrow[from=4-3, to=3-4]
\end{tikzcd}\hspace{-3pt}\]
We wish to show that the front face is cartesian.
We first argue that the four slanted arrows are equivalences.
The lower-right one is the identity.
For $k=1,n-1$, the slanted map
\[
\abs{\Zig_{[k]}\pr{\cat C,\cat W}}\to\abs{\Zig_{[-1,k,-1]}\pr{\cat C,\cat W}}
\]
inserts identity weak equivalences at both endpoints.
Restriction to the central $[k]$-shaped string gives a homotopy inverse after realization, as is seen from the natural transformations induced by the two outer weak equivalences.
The remaining slanted map factors as
\begin{align*}
\abs{\Zig_{[n]}\pr{\cat C,\cat W}}
&\to\abs{\Zig_{[-1,n,-1]}\pr{\cat C,\cat W}}\\
&\to\abs{\Zig_{[-1,1,-1,n-1,-1]}\pr{\cat C,\cat W}}\\
&\to\abs{\Zig_{[-1,1,-1][-1,n-1,-1]}\pr{\cat C,\cat W}}.
\end{align*}
The first map is an equivalence by the preceding argument, the second by the two-sided-fraction hypothesis, and the third by composing the two consecutive left-pointing arrows at the marked junction.

Consequently, it suffices to show that the back face is cartesian. By Quillen's Theorem B (Proposition \ref{prop:ThmB}), this will follow once we show that each morphism $w\from y\to y'$ in $\cat W$ induces an equivalence
\[
\abs{\Zig_{[-1,1,-1]}\pr{\cat C,\cat W}_{y'}}\xrightarrow{\simeq}\abs{\Zig_{[-1,1,-1]}\pr{\cat C,\cat W}_{y}}.
\]
For this, we consider the commutative diagram 
\[\begin{tikzcd}
	{\Zig_{[-1,1,-1]}(\cat{C},\cat{W})_{y'}} && {\Zig_{[-1,1,-1]}(\cat{C},\cat{W})_{y}} \\
	& {\cat{W}.}
	\arrow[from=1-1, to=1-3]
	\arrow["\dom"', from=1-1, to=2-2]
	\arrow["\dom", from=1-3, to=2-2]
\end{tikzcd}\]
The slanted arrows are cocartesian fibrations.
For each $x\in\cat W$, the induced map between the fibers over $x\in\cat W$ is a weak homotopy equivalence by Theorem \ref{thm:fraction} and Corollary \ref{cor:ham_equivalence}.
Thus, Corollary \ref{cor:fiberwise_pullback} shows that the horizontal map is a weak homotopy equivalence, as required.

The assertion about completeness now follows from Remark \ref{rem:saturated}.
\end{proof}

\section{\label{sec:criteria}Criteria and examples for fractions}

This section gives criteria and examples for supporting fractions.
In the first two subsections, we provide sufficient conditions for relative $\infty$-categories to support fractions (Subsections \ref{subsec:general} and \ref{subsec:gabriel_zisman}).
We then provide a few sources of relative $\infty$-categories that support fractions.
The first source is relative $\infty$-categories satisfying quasicategorical calculus of fractions in the sense of Carranza--Kapulkin--Lindsey (Subsection \ref{subsec:quasi}); their calculus already extends the classical Gabriel--Zisman calculus of fractions to the $\infty$-categorical setting.
At the 1-categorical level, a comparison between the two styles has been observed by Dwyer--Kan themselves \cite[7.2]{DK80_2}.
The second source is $\infty$-categories of fibrant objects (Subsection \ref{subsec:fibrant}); again, the fact that their 1-categorical counterparts admit a calculus of fractions in the Dwyer--Kan sense has long been known to experts. 
\begin{rem}
Some of the results in this section are stated only for left or right fractions.
Their dual versions for right or left fractions are of course true, but we will not state them to reduce redundancy.
\end{rem}

\subsection{\label{subsec:general}Functorial criteria for fractions}

We start with an $\infty$-categorical generalization of \cite[Propositions 8.1, 8.2]{DK80_2}.

\begin{defn}
We write $\msquare$ for the relative poset with underlying poset $[1]\times[1]$ and with weak equivalences as depicted in the symbol.
We also denote by
\[
	\mspan , \quad \mcospan \quad \subseteq \quad \msquare
\]
the two relative subposets obtained by deleting the final object and initial object, respectively.
\end{defn}

\begin{prop}
\label{prop:8.1}Let $\pr{\cat C,\cat W}$ be a relative $\infty$-category such that $\cat W$ satisfies the two-out-of-three property.

\begin{enumerate}
\item If the functor
\[
\Fun_{\rel}\pr{\msquare,\pr{\cat C,\cat W}}\to\Fun_{\rel}\pr{\mspan,\pr{\cat C,\cat W}}
\]
admits a section, then $\pr{\cat C,\cat W}$ supports left fractions.
\item If the functor
\[
\Fun_{\rel}\pr{\msquare,\pr{\cat C,\cat W}}\to\Fun_{\rel}\pr{\mcospan,\pr{\cat C,\cat W}}
\]
admits a section, then $\pr{\cat C,\cat W}$ supports right fractions.
\item Suppose there are wide subcategories $\cat U,\cat V\subset\cat W$ satisfying the following conditions:

\begin{enumerate}
\item Every morphism in $\cat W$ can be functorially factored as a map in $\cat U$ followed by a map in $\cat V$.
\item The pair $(\cat{C},\cat{U})$ satisfies the hypothesis of (1).
\item The pair $(\cat{C},\cat{V})$ satisfies the hypothesis of (2).
\end{enumerate}
Then $\pr{\cat C,\cat W}$ supports two-sided fractions.

\end{enumerate}
\end{prop}

\begin{proof}
The proofs of all parts are similar, so we will prove (3).
We must show that, for each $i,j>0$ and each pair of objects $x,y\in\cat C$, the functor
\[
\theta\from\Zig_{[-1,i+j,-1]}\pr{\cat C,\cat W}_{x,y}\to\Zig_{[-1,i,-1,j,-1]}\pr{\cat C,\cat W}_{x,y}
\]
is a weak homotopy equivalence.
To this end, we define a functor in the opposite direction, which maps an object depicted in the top row to the bottom row in the diagram below: 
\[\begin{tikzcd}
	x & {a_0} & \cdots & {a_i} & {b_0} & \cdots & {b_j} & y \\
	x & {a'_0} & \cdots & {a'_i} & {b_0'} & \cdots & {b'_j} & y
	\arrow["\sim"', from=1-2, to=1-1]
	\arrow[from=1-2, to=1-3]
	\arrow[from=1-3, to=1-4]
	\arrow["\sim"', from=1-5, to=1-4]
	\arrow[from=1-5, to=1-6]
	\arrow["u", from=1-5, to=2-5]
	\arrow[from=1-6, to=1-7]
	\arrow["u", from=1-7, to=2-7]
	\arrow["\sim"', from=1-8, to=1-7]
	\arrow[equal, from=1-8, to=2-8]
	\arrow[equal, from=2-1, to=1-1]
	\arrow[from=2-2, to=1-2]
	\arrow["\sim"', from=2-2, to=2-1]
	\arrow[from=2-2, to=2-3]
	\arrow[from=2-3, to=1-3]
	\arrow[from=2-3, to=2-4]
	\arrow["v", from=2-4, to=1-4]
	\arrow[equal, from=2-5, to=2-4]
	\arrow[from=2-5, to=2-6]
	\arrow[from=2-6, to=2-7]
	\arrow["\sim"', from=2-8, to=2-7]
\end{tikzcd}\]
The diagram needs some explanation: We first use (a) to factor the map $a_{i}\leftarrow b_{0}$ as $vu$, where $u$ and $v$ are morphisms in $\cat U$ and $\cat V$, respectively.
We then recursively apply (b) and (c) to fill in the rest of the diagram.
The two-out-of-three property ensures that the functorial factorization and the chosen sections preserve natural transformations that are pointwise in $\cat W$: at each factorization or square-completion step, the new component composes with a known weak equivalence to a weak equivalence, and is therefore itself weak.
Since all the choices are functorial, the construction defines a functor
\[
F\from\Zig_{[-1,i,-1,j,-1]}\pr{\cat C,\cat W}_{x,y}\to\Zig_{[-1,i+j,-1]}\pr{\cat C,\cat W}_{x,y}.
\]

We next construct the two inverse homotopies.
First, the displayed diagram can be split through the intermediate row
\[\begin{tikzcd}[column sep = small]
	x & {a_0} & \cdots & {a_i} & {b'_0} & \cdots & {b'_j} & y
	\arrow["\sim"', from=1-2, to=1-1]
	\arrow[from=1-2, to=1-3]
	\arrow[from=1-3, to=1-4]
	\arrow["\sim"', "v", from=1-5, to=1-4]
	\arrow[from=1-5, to=1-6]
	\arrow[from=1-6, to=1-7]
	\arrow["\sim"', from=1-8, to=1-7]
\end{tikzcd}\]
which agrees with the top row to the left of $v$ and with the bottom row to the right of it.
This gives a zig-zag of natural transformations
\[
\id\longrightarrow H\longleftarrow\theta F
\]
whose components belong to $\cat W$, and hence a homotopy $\abs{\theta}\abs{F}\simeq\id$ after realization.

For the other composite, consider the restriction of the construction to the image of $\theta$.
The newly inserted weak equivalence is then an identity.
If its chosen factorization is
\[
z\xrightarrow{u}c\xrightarrow{v}z,
\qquad vu=\id_z,
\]
then the maps $u$ and $v$ give natural transformations, pointwise in $\cat W$, between this factorization and the trivial factorization of $\id_z$.
Likewise, a chosen square completion along an identity weak equivalence is connected to the trivial square completion by a natural transformation whose only non-identity component is the newly created weak equivalence.
Because every stage of the construction is functorial, these comparisons can be propagated through the later stages.
Replacing the chosen factorization and square completions one at a time by the trivial ones therefore gives a zig-zag of natural transformations between $F\theta$ and $\id$.
It follows that $\abs{F}$ is an inverse of $\abs{\theta}$.
\end{proof}

\begin{rem}
More generally, the two-out-of-three hypothesis in Proposition \ref{prop:8.1} can be omitted if the relevant sections, and in part (3) also the functorial factorization, preserve natural transformations that are pointwise in $\cat W$.
\end{rem}

\begin{example}
Every model category with functorial factorizations satisfies the hypothesis of Proposition \ref{prop:8.1} (3), by taking $\cat U$ and $\cat V$ to be the classes of acyclic cofibrations and acyclic fibrations, respectively.
Consequently, every such model category supports two-sided fractions.
In fact, every model category supports two-sided fractions, without any assumption on functorial factorizations \cite[7.2(iii) and 8.1]{DK80_3}.
See also \cite[Lemma 8.2]{MG21} and \cite{BK} for analogous results for model $\infty$-categories and partial model categories, respectively, and Subsection \ref{subsec:fibrant} for the corresponding one-sided result for Brown's categories of fibrant objects and their $\infty$-categorical analogues.
\end{example}

\begin{example}
Let $\pr{\cat C,\cat W}$ be a relative $\infty$-category.
If $\cat W$ satisfies the two-out-of-three property, $\cat C$ admits pullbacks, and morphisms in $\cat W$ are stable under pullback, then Proposition \ref{prop:8.1} shows that $\pr{\cat C,\cat W}$ supports right fractions.
In particular, Corollary \ref{cor:hammock_anima_is_anima_of_loc} and Theorem \ref{thm:fraction} give
\[
\Map_{\cat C[\cat W^{-1}]}\pr{x,y}\simeq\abs{\Zig_{[-1,1]}\pr{\cat C,\cat W}_{x,y}}.
\]
This was obtained by a different method by Cisinski \cite[Theorem 7.2.8, Remark 7.2.10, and Theorem 7.2.16]{HCHA}.
\end{example}

\subsection{\label{subsec:gabriel_zisman}A Gabriel--Zisman criterion}

Dwyer--Kan's theory of fractions is predated by Gabriel--Zisman's \cite{GZ67}, which describes hom-sets of $1$-categorical localizations.
As observed by Dwyer--Kan \cite[7.2]{DK80_2} themselves, their theory subsumes Gabriel--Zisman's.
The goal of this subsection is to establish an $\infty$-categorical analog of this.
To state the criterion, we need to introduce a bit of notation.

\begin{defn}
Let $\pr{\cat C,\cat W}$ be a relative $\infty$-category.
For a relative $\infty$-category $K$, write
\[
\Fun_{\rel}\pr{K,\pr{\cat C,\cat W}}{}^{\cat W}\subset\Fun_{\rel}\pr{K,\pr{\cat C,\cat W}}
\]
for the wide subcategory whose morphisms are the natural transformations that are pointwise in $\cat W$.
For each pair of morphisms $\smash[b]{a'\xleftarrow[u]{\sim}a\xrightarrow[f]{}b}$ in $\cat C$, define $\Sq_{\cat W}\pr{u,f}$ by the pullback square
\[\begin{tikzcd}
	{\Sq_{\cat W}\pr{u,f}} & {\Fun_{\rel}\pr{\msquare,\pr{\cat C,\cat W}}{}^{\cat W}} \\
	{\{a'\xleftarrow[u]{\sim}a\xrightarrow[f]{}b\}} & {\Fun_{\rel}\pr{\mspan,\pr{\cat C,\cat W}}{}^{\cat W}.}
	\arrow[from=1-1, to=1-2]
	\arrow["\lrcorner"{description, pos=0.1}, draw=none, from=1-1, to=2-2]
	\arrow[from=1-1, to=2-1]
	\arrow[from=1-2, to=2-2]
	\arrow[from=2-1, to=2-2]
\end{tikzcd}\]
Thus, its objects are commutative squares of the form
\[\begin{tikzcd}
	a & b \\
	{a'} & {b',}
	\arrow["f", from=1-1, to=1-2]
	\arrow["u"', "\sim", from=1-1, to=2-1]
	\arrow["w"', "\sim", from=1-2, to=2-2]
	\arrow[from=2-1, to=2-2]
\end{tikzcd}\]
where $w\in\cat W$, and its morphisms are maps between such squares whose component at $b'$ belongs to $\cat W$.
\end{defn}

\begin{prop}
\label{prop:span_hCLF}Let $\pr{\cat C,\cat W}$ be a relative $\infty$-category.
Suppose that, for each pair of morphisms $a'\xleftarrow[\liftlabel{u}]{\sim}a\xrightarrow[\liftlabel{f}]{}b$ in $\cat C$, the $\infty$-category $\Sq_{\cat W}\pr{u,f}$ is weakly contractible.
Then $\pr{\cat C,\cat W}$ supports left fractions.
\end{prop}

\begin{example}
The hypothesis of Proposition \ref{prop:span_hCLF} is satisfied by any relative $1$-category $\pr{\cat C,\cat W}$ which admits a calculus of left fractions in the sense of Gabriel--Zisman and for which $\cat W$ satisfies the two-out-of-three property.
Indeed, each $\Sq_{\cat W}\pr{u,f}$ is filtered in this case.
\end{example}

\begin{example}
Suppose that $\cat W$ satisfies the two-out-of-three property, that $\cat C$ admits pushouts, and that morphisms in $\cat W$ are stable under pushout in $\cat C$.
Then the hypothesis of Proposition \ref{prop:span_hCLF} is satisfied.
Indeed, the pushout square is an initial object of each $\Sq_{\cat W}\pr{u,f}$.
\end{example}

The proof of Proposition \ref{prop:span_hCLF} relies on the following observation:

\begin{rem}
\label{rem:square_completion_fiber}In the situation of the proposition, the $\infty$-category $\Sq_{\cat W}\pr{u,f}$ can be identified with the fiber over $u$ of the domain functor
\[
\dom\from\Zig_{[1]}\pr{\cat C,\cat W}_{f/}\to\cat W_{a/}.
\]
To see this, consider the following commutative diagram:
\[
\begin{tikzcd}[column sep=small]
	{\Sq_{\cat W}\pr{u,f}} & {\Zig_{[1]}\pr{\cat C,\cat W}_{f/}} & {\Fun_{\rel}\pr{\msquare,\pr{\cat C,\cat W}}{}^{\cat W}} \\
	{\{u\}} & {\cat W_{a/}} & {\Fun_{\rel}\pr{\mspan,\pr{\cat C,\cat W}}{}^{\cat W}} \\
	& {\{f\}} & {\Fun_{\rel}\pr{\marrow,\pr{\cat C,\cat W}}^{\cat W}.}
	\arrow[from=1-1, to=1-2]
	\arrow[from=1-1, to=2-1]
	\arrow[from=1-2, to=1-3]
	\arrow["\dom", from=1-2, to=2-2]
	\arrow[from=1-3, to=2-3]
	\arrow[from=2-1, to=2-2]
	\arrow[from=2-2, to=2-3]
	\arrow[from=2-2, to=3-2]
	\arrow[from=2-3, to=3-3]
	\arrow[from=3-2, to=3-3]
\end{tikzcd}
\]
The lower-right square is cartesian by inspection, while the right outer rectangle is cartesian by identifying the upper right corner of the diagram with the arrow category of $\Zig_{[1]}\pr{\cat C,\cat W}$. Since the upper outer rectangle is cartesian by definition, the pasting law of cartesian squares implies that the upper-left square is cartesian, as claimed.
\end{rem}

\begin{proof}
[Proof of Proposition \ref{prop:span_hCLF}]
Let $x,y\in\cat C$ be a pair of objects, and let $i,j\geq0$.
We must show that the map
\[
\theta\from\Zig_{[i+j,-1]}\pr{\cat C,\cat W}_{x,y}\to\Zig_{[i,-1,j,-1]}\pr{\cat C,\cat W}_{x,y}
\]
is a weak homotopy equivalence.
If $j=0$, composing the two left-pointing arrows on the right gives a homotopy inverse of $\abs{\theta}$.
Thus, we may assume that $j>0$.
In this case, we show $\theta$ is in fact final.

Let $\zeta\in\Zig_{[i,-1,j,-1]}\pr{\cat C,\cat W}_{x,y}$ be an arbitrary object, depicted as
\[
x\to a_{1}\to\cdots\to a_{i}\xleftarrow[v]{\sim}b_{0}\to\cdots\to b_{j}\xleftarrow[w]{\sim}y.
\]
We must show that the relative slice category $\pr{\Zig_{[i+j,-1]}\pr{\cat C,\cat W}_{x,y}}_{\zeta/}$ of $\theta$ is weakly contractible.
We prove this by induction on $j$, uniformly in $\zeta$.
For this, observe that $\Zig_{[i,-1,j,-1]}\pr{\cat C,\cat W}_{x,y}$ can be decomposed as
\begin{align*}
\Zig_{[i,-1,j,-1]}\pr{\cat C,\cat W}_{x,y} & \simeq\ps{\Zig_{[i,-1]}\pr{\cat C,\cat W}}x\times_{\cat W}\Zig_{[j]}\pr{\cat C,\cat W}\times_{\cat W}\Zig_{[-1]}\pr{\cat C,\cat W}_{y}\\
 & \simeq\ps{\Zig_{[i,-1]}\pr{\cat C,\cat W}}x\times_{\cat W}\Zig_{[j]}\pr{\cat C,\cat W}\times_{\cat W}\cat W_{y/}.
\end{align*}
Write
\[
\zeta_0=\pr{x\to a_{1}\to\cdots\to a_{i}\xleftarrow[v]{\sim}b_{0}}
\qquad\text{and}\qquad
\zeta_1=\pr{b_{0}\to\cdots\to b_{j}}.
\]
An object of the relative slice is a pointwise weak equivalence from $\zeta$ to a zig-zag in the image of $\theta$, which we may describe by cutting it at $b_0$ and $b_j$.
The preceding decomposition therefore gives
\begin{align*}
\pr{\Zig_{[i+j,-1]}\pr{\cat C,\cat W}_{x,y}}_{\zeta/}
\simeq & \pr{\ps{\Zig'_{[i,-1]}\pr{\cat C,\cat W}}x}_{\zeta_0/}\times_{\cat W_{b_{0}/}}\pr{\Zig_{[j]}\pr{\cat C,\cat W}_{\zeta_1/}}\times_{\cat W_{b_{j}/}}\pr{\cat W_{y/}}_{w/}\\
\simeq & \pr{\ps{\Zig'_{[i,-1]}\pr{\cat C,\cat W}}x}_{\zeta_0/}\times_{\cat W_{b_{0}/}}\pr{\Zig_{[j]}\pr{\cat C,\cat W}_{\zeta_1/}},
\end{align*}
where $\ps{\Zig'_{[i,-1]}\pr{\cat C,\cat W}}x\subset\ps{\Zig{}_{[i,-1]}\pr{\cat C,\cat W}}x$ denotes the essential image of $\ps{\Zig{}_{[i]}\pr{\cat C,\cat W}}x$, i.e., the full subcategory of those zig-zags whose last arrow is an equivalence.
Here, the first factor records the part to the left of $b_0$ and involves $\Zig'_{[i,-1]}$ because $\theta$ inserts an identity as the middle weak equivalence, while the second records the deformation of the chain $b_0\to\cdots\to b_j$.
The final factor in the first line disappears by the equivalence $\pr{\cat W_{y/}}_{w/}\simeq\cat W_{b_j/}$.
The domain projection
\[
\Zig_{[j]}\pr{\cat C,\cat W}_{\zeta_1/}\to\cat W_{b_{0}/}
\]
is a cartesian fibration and hence proper by the dual of \cite[Proposition 4.1.2.15]{HTT}.
Moreover, the $\infty$-category $\pr{\ps{\Zig'_{[i,-1]}\pr{\cat C,\cat W}}x}_{\zeta_0/}$ has an initial object depicted as
\[\begin{tikzcd}
	x & {a_1} & \cdots & {a_i} & {b_0} \\
	x & {a_1} & \cdots & {a_i} & {a_i.}
	\arrow[from=1-1, to=1-2]
	\arrow[equal, from=1-1, to=2-1]
	\arrow[from=1-2, to=1-3]
	\arrow[equal, from=1-2, to=2-2]
	\arrow[from=1-3, to=1-4]
	\arrow[equal, from=1-4, to=2-4]
	\arrow["\sim"', "v", from=1-5, to=1-4]
	\arrow["\sim", "v"', from=1-5, to=2-5]
	\arrow[from=2-1, to=2-2]
	\arrow[from=2-2, to=2-3]
	\arrow[from=2-3, to=2-4]
	\arrow[equal, from=2-5, to=2-4]
\end{tikzcd}\]
To verify initiality, let $\zeta_0\to\eta$ be an arbitrary object of the slice.
A morphism from the displayed object to $\zeta_0\to\eta$ is forced on every vertex except the last one.
The remaining component is uniquely obtained by lifting across the final arrow of $\eta$, which is an equivalence; it belongs to $\cat W$ because $\cat W$ contains all equivalences and is closed under composition.
For $1\leq k\leq j$, set
\[
T_k=
\{v\}\times_{\cat W_{b_{0}/}}
\pr{\Zig_{[k]}\pr{\cat C,\cat W}_{b_{0}\to\cdots\to b_{k}/}}.
\]
The displayed initial object lies over $v\in\cat W_{b_0/}$.
Since the domain projection is proper, the base change of the inclusion of this initial object is the initial functor
\begin{align*}
T_j\hookrightarrow
\pr{\ps{\Zig'_{[i,-1]}\pr{\cat C,\cat W}}x}_{\zeta_0/}
\times_{\cat W_{b_{0}/}}
\pr{\Zig_{[j]}\pr{\cat C,\cat W}_{\zeta_1/}}.
\end{align*}
Hence, it suffices to show that $T_j$ is weakly contractible.

Remark \ref{rem:square_completion_fiber} gives an equivalence
\[
T_1\simeq\Sq_{\cat W}\pr{v,b_0\to b_1},
\]
which is weakly contractible by hypothesis.
This proves the claim when $j=1$.
For $j>1$, restriction to the first arrow gives a cartesian fibration
\[
p\from T_j\to T_1.
\]
A cartesian lift is obtained by leaving the remaining objects unchanged and precomposing the first arrow in the remaining tail.
The fibers of $p$ are the analogous completion categories for the remaining tail, and hence are weakly contractible by induction.
Since $T_1$ is weakly contractible, Quillen's Theorem B (Proposition \ref{prop:ThmB}) shows that $T_j$ is weakly contractible.
Thus $\theta$ is final, as required.
\end{proof}

\subsection{\label{subsec:quasi}Comparison with quasicategorical fractions}

We next compare Proposition \ref{prop:span_hCLF} with the quasicategorical calculus of fractions developed by Carranza--Kapulkin--Lindsey \cite{CKL25}.
For relative $1$-categories, their definition reduces to Gabriel--Zisman's classical definition.
The main result of \cite{CKL25} asserts that a relative quasicategory admitting their calculus of fractions admits an explicit combinatorial model of localization (as a simplicial set).
Using this explicit model, they give a description of the mapping animae of localization in terms of fractions.

The main result of this subsection shows that their extension axioms imply the hypothesis of Proposition \ref{prop:span_hCLF}: More precisely, they imply that each $\infty$-category $\Sq_{\cat W}\pr{u,f}$ is filtered, and hence weakly contractible.
It follows that every relative $\infty$-category admitting their calculus supports left fractions (Theorem \ref{thm:CLF}).
Combining this with Theorem \ref{thm:fraction} gives a shorter, perhaps more conceptual proof of the aforementioned mapping-anima description of localizations (Remark \ref{rem:map_CLF}).

We start by recalling the definition of calculi in the sense of \cite{CKL25}.

\begin{defn}
\cite[Definition 4.4]{CKL25}\label{def:CLF} For each pair of integers $n\geq2$ and $0<k\leq n$, we define a relative poset $\cat P_{n,k}$ as follows: Its underlying poset is the poset of subsets of $[n]$ that contain $k$, ordered by inclusion.
A morphism $S\to T$ is a weak equivalence if and only if they share the same maxima.

We also define a relative poset $\cat P'_{n,k}$ similarly, by considering the set of \textit{proper} subsets of $[n]$ that contain $k$.

A relative $\infty$-category $\pr{\cat C,\cat W}$ is said to admit \textit{CLF} (\textit{calculus of left fractions}) if for every $n\geq2$ and $0<k\leq n$, every relative functor $\cat P'_{n,k}\to\pr{\cat C,\cat W}$ extends to $\cat P_{n,k}$ (up to equivalence).
\end{defn}

\begin{example}
We can depict $\cat P_{2,1}$ as 
\[\begin{tikzcd}
	\bullet & \bullet \\
	\bullet & \bullet.
	\arrow["\sim", from=1-1, to=1-2]
	\arrow[from=1-1, to=2-1]
	\arrow[from=1-2, to=2-2]
	\arrow["\sim"', from=2-1, to=2-2]
\end{tikzcd}\]
More generally, we have $\cat P_{n,n-1}\simeq\pr{[1]^{n},[1]^{n-1}\times\{0,1\}}$, where $\{0,1\}$ is the discrete category with two objects.
\end{example}

\begin{rem}
In \cite{CKL25}, the definition of CLF is stated using quasicategories equipped with a designated subset of morphisms, which are required to be ``weakly closed under composition'' and contain all identity morphisms.
If we interpret $\infty$-categories as quasicategories, then the assumption on weak closure adds extra generality in the definition.
This does not, however, seem to make much difference in practice.

Also, in \cite{CKL25}, the definition of CLF requires that every relative functor $\cat P'_{n,k}\to\pr{\cat C,\cat W}$ admits a strict extension to $\cat P_{n,k}$ as a map of simplicial sets.
In the setting at hand, this requirement comes for free from Definition \ref{def:CLF} because we assume that $\cat W$ is a wide subcategory of $\cat C$.
\end{rem}

Here is the main result of this subsection:

\begin{thm}
\label{thm:CLF}Every relative $\infty$-category admitting CLF supports left fractions.
\end{thm}

We will prove Theorem \ref{thm:CLF} from the following more general result:

\begin{thm}
\label{thm:CLFgeneral}Let $\pr{\cat C,\cat W}$ be a relative $\infty$-category.
Suppose that, for every $n\geq2$, every relative functor $\cat P'_{n,n-1}\to\pr{\cat C,\cat W}$ extends to $\cat P_{n,n-1}$ (up to equivalence).
Then $\pr{\cat C,\cat W}$ supports left fractions. 
\end{thm}

\begin{rem}
\label{rem:map_CLF}Combining Theorem \ref{thm:CLFgeneral} with Corollary \ref{cor:hammock_anima_is_anima_of_loc} and Theorem \ref{thm:fraction}, we find that for a relative $\infty$-category $\pr{\cat C,\cat W}$ satisfying the hypothesis of Theorem \ref{thm:CLFgeneral}, we have
\[
\Map_{\cat C[\cat W^{-1}]}\pr{x,y}\simeq\abs{\Zig_{[1,-1]}\pr{\cat C,\cat W}_{x,y}}.
\]
In the special case of relative $\infty$-categories admitting CLF, this equivalence was obtained in \cite[Corollary 10.10]{CKL25} by using an explicit combinatorial argument.
\end{rem}

The remainder of this subsection is devoted to the proof of Theorem \ref{thm:CLFgeneral}.

\begin{notation}
For each $n\geq0$, we write $D\pr n$ for the poset of nonempty subsets of $\{0,\dots,n\}$, ordered by inclusion.
We also write $\partial D\pr n\subset D\pr n$ for the full subposet consisting of proper, nonempty subsets of $\{0,\dots,n\}$. 
\end{notation}

\begin{lem}
\label{lem:filtered}\cite[Lemma 11.3]{CKL25} Let $\cat X$ be an $\infty$-category.
The following conditions are equivalent:

\begin{enumerate}
\item $\cat X$ is filtered.
\item For every $n\geq0$, the functor $\Fun\pr{D\pr n,\cat X}\to\Fun\pr{\partial D\pr n,\cat X}$ is essentially surjective.
\end{enumerate}
\end{lem}

\begin{notation}
We write $\Lambda=\Lambda^{2}_{0}$ for the walking span, i.e., the poset with three elements $0,1,2$, with relation generated by $0\leq1$ and $0\leq2$.
\end{notation}

\begin{construction}
\label{const:chi}Let $n\geq0$.
We define a functor
\[
\chi\from\{\emptyset\}\star D\pr{n+1}\to\Lambda\star D\pr n
\]
by
\[
\chi\pr S=\begin{cases}
0 & \text{if }S=\emptyset,\\
1 & \text{if }S\neq\emptyset\text{ and }n+1\not\in S,\\
2 & \text{if }S=\{n+1\},\\
S\setminus\{n+1\} & \text{if }S\supsetneq\{n+1\}.
\end{cases}
\]
We write $\chi_{\partial}\from\{\emptyset\}\star\partial D\pr{n+1}\to\Lambda\star\partial D\pr n$ for the restriction of $\chi$.
\end{construction}

\begin{lem}
\label{lem:chi_final}Let $n\geq0$.
The functor
\[
\chi_{\partial}\from\{\emptyset\}\star\partial D\pr{n+1}\to\Lambda\star\partial D\pr n
\]
is final.
Consequently, the square
\[
\begin{tikzcd}
	{\{\emptyset \}\star \partial D(n+1)} & {\Lambda \star \partial D(n)} \\
	{\{\emptyset \}\star D(n+1)} & {\Lambda \star D(n)}
	\arrow["{\chi_{\partial}}", from=1-1, to=1-2]
	\arrow[hook, from=1-1, to=2-1]
	\arrow[hook, from=1-2, to=2-2]
	\arrow["\chi", from=2-1, to=2-2]
\end{tikzcd}
\]
is a pushout of $\infty$-categories.
\end{lem}

\begin{proof}
Write
\[
\cat A=\{\emptyset\}\star\partial D\pr{n+1}
\qquad\text{and}\qquad
\cat B=\Lambda\star\partial D\pr n.
\]
By Quillen's Theorem A, it suffices to show that for each $b\in\cat B$, the comma category $\cat A\times_{\cat B}\cat B_{b/}$ is weakly contractible. We may identify this comma category with the poset
\[
\cat A_b=\{S\subsetneq[n+1]\mid b\leq\chi_{\partial}\pr S\}.
\]
If $b=0$, then $\cat A_b=\cat A$ has initial object $\emptyset$.
If $b=2$, then $\cat A_b$ has initial object $\{n+1\}$.
If $b=J\in\partial D\pr n$, then $\cat A_b$ has initial object $J\cup\{n+1\}$.

It remains to consider $b=1$.
In this case,
\[
\cat A_1=\{S\subsetneq[n+1]\mid S\neq\emptyset,\{n+1\}\}.
\]
The functor
\[
r\from\cat A_1\to D\pr n,
\qquad
r\pr S=S\setminus\{n+1\},
\]
is a retraction of the inclusion $i\from D\pr n\to\cat A_1$, and the inclusions $ir\pr S\subseteq S$ define a natural transformation $ir\to\id_{\cat A_1}$.
It follows that $\abs{\cat A_1}\simeq\abs{D\pr n}$ is contractible.
Thus, $\cat A_b$ is weakly contractible for every $b\in\cat B$, and so $\chi_{\partial}$ is final.

For the final claim, note that there are canonical identifications
\[
\{\emptyset\}\star D\pr{n+1}\simeq\cat A^{\triangleright}
\qquad\text{and}\qquad
\Lambda\star D\pr n\simeq\cat B^{\triangleright}.
\]
By the characterization of finality in terms of cocones \cite[Proposition 4.1.1.8]{HTT}, restriction along $\chi_{\partial}$ induces an equivalence $\cat E_{F/}\xrightarrow{\sim}\cat E_{F\chi_{\partial}/}$ for every $\infty$-category $\cat E$ and every functor $F\from\cat B\to\cat E$. In particular, the canonical square
\[
\begin{tikzcd}
	\cat A & \cat B \\
	\cat A^{\triangleright} & \cat B^{\triangleright}
	\arrow["{\chi_{\partial}}", from=1-1, to=1-2]
	\arrow[hook, from=1-1, to=2-1]
	\arrow[hook, from=1-2, to=2-2]
	\arrow[from=2-1, to=2-2]
\end{tikzcd}
\]
is a pushout, which proves the final assertion.
\end{proof}

\begin{proof}
[Proof of Theorem \ref{thm:CLFgeneral}]
We will use Proposition \ref{prop:span_hCLF}.
Let $a'\xleftarrow[v]{\sim}a\xrightarrow[f]{}b$ be a pair of morphisms in $\cat C$, where $v$ is a weak equivalence.
We wish to show that the $\infty$-category $\Sq_{\cat W}\pr{v,f}$ is weakly contractible.
According to Lemma \ref{lem:filtered}, it suffices to show that for every $n\geq0$, every functor $F\from\partial D\pr n\to\Sq_{\cat W}\pr{v,f}$ extends to $D\pr n$.

The functor $F$ determines a map $F'\from\Lambda\star\partial D\pr n\to\cat C$: The three objects of $\Lambda$ are mapped to $a,a'$, and $b$, respectively, the map $0\to1$ is mapped to $v$, and each $J\in\partial D\pr n$ is mapped to the lower-right vertex of the square $F\pr J$.
In particular, $F'$ carries the maps $2\to J$ and $J\to K$ for $J\subset K$ to weak equivalences.
Precomposing with the functor $\chi_{\partial}$ of Construction \ref{const:chi}, we obtain a diagram 
\[
F'\circ\chi_{\partial}\from\{\emptyset\}\star\partial D\pr{n+1}\to\cat C.
\]
Now we can identify the posets $\{\emptyset\}\star\partial D\pr{n+1}$ and $\{\emptyset\}\star D\pr{n+1}$ with the underlying posets of $\cat P'_{n+2,n+1}$ and $\cat P_{n+2,n+1}$.
Explicitly, the identification is given by the functor $S\mapsto\partial_{n+1}\pr S\cup\{n+1\}$, where $\partial_{n+1}\from[n+1]\to[n+2]$ is the injective poset map that skips $n+1$.
Under this identification, the functor $F'\circ\chi_{\partial}$ carries every weak equivalence of $\cat P'_{n+2,n+1}$ to a weak equivalence.
Thus, our hypothesis gives a dashed filler in the diagram 
\[\begin{tikzcd}
	{\{\emptyset \}\star \partial D(n+1)} & {\Lambda \star \partial D(n)} & \\
	{\{\emptyset \}\star D(n+1)} \\
	&& {\cat{C}}
	\arrow["{\chi_\partial}", from=1-1, to=1-2]
	\arrow[hook, from=1-1, to=2-1]
	\arrow["{F'}", from=1-2, to=3-3]
	\arrow[dashed, from=2-1, to=3-3]
\end{tikzcd}\]
that carries every morphism of the form $S\subset T$ with $n+1\in S$ to a weak equivalence.
Lemma \ref{lem:chi_final} then gives an extension $\Lambda\star D\pr n\to\cat C$ of $F'$.
Its map $2\to[n]$ is induced by $\{n+1\}\subset[n+1]$, and its maps $J\to[n]$ are induced by $J\cup\{n+1\}\subset[n+1]$.
All these maps are weak equivalences by the preceding property of the dashed filler, so this extension determines a map
\[
D\pr n\to\Sq_{\cat W}\pr{v,f}
\]
extending $F$.
\end{proof}

\begin{rem}
\label{rem:CLFgeneral_converse_false}In the situation of Theorem \ref{thm:CLFgeneral}, suppose in addition that $\pr{\cat C,\cat W}$ has the two-out-of-three property.
Then for every object $c\in\cat C$, the $\infty$-category $\cat W_{c/}$ is filtered.
(This is an easy consequence of Lemma \ref{lem:filtered}.)
Thus, the equivalence
\[
\Map_{\cat C[\cat W^{-1}]}\pr{x,y}\simeq\colim_{y\to y'\in\cat W_{y/}}\Map_{\cat C}\pr{x,y'}
\]
of Remark \ref{rem:fraction_naturality} shows that if $\cat C$ has finite colimits, then the localization functor $\gamma\from\cat C\to\cat C[\cat W^{-1}]$ preserves them.
(In fact, $\cat C[\cat W^{-1}]$ is also finitely cocomplete; see the argument of \cite[Theorem 11.7]{CKL25}.) 

In particular, the converse of Theorem \ref{thm:CLFgeneral} is \textit{false}, as there are many relative $\infty$-categories that support left fractions but whose localization functors do not preserve all finite colimits that exist in the underlying $\infty$-category (such as many naturally occurring Brown's categories of cofibrant objects).
\end{rem}

\subsection{\label{subsec:fibrant}\texorpdfstring{$\infty$-Categories}{∞-Categories} of fibrant objects}

Brown's \textit{categories of fibrant objects} \cite{Bro73} are relative categories equipped with another class of maps, called fibrations, subject to some axioms.
Their utility comes from the fact that their ($\infty$-categorical) localizations are finitely complete, and that pullbacks and products in the localization can be constructed before passing to the localization \cite[Proposition 7.5.6]{HCHA}. 

Crucial in the development of categories of fibrant objects is the property that every morphism in their localization can be written as a right fraction \cite[Theorem 1]{Bro73}.
By now, it is well-known that they support right fractions.\footnote{An oft-cited reference for this fact (assuming a functorial factorization) is \cite[\S 3.6.2]{NSS15}, but the proof of loc. cit. lacks a subtle but crucial ingredient: For right fractions, Lemma 3.63 needs to be proved for $j=0$ too, but the proof is not supplied and is non-trivial.}

In \cite{HCHA}, Cisinski introduced \textit{$\infty$-categories of fibrant objects}, an $\infty$-categorical generalization of Brown's definition.
He went on to observe that they support many existing results for Brown's 1-categorical counterpart.
In this subsection, we will continue this trend by showing that $\infty$-categories of fibrant objects support right fractions (Theorem \ref{thm:fibrationcats}).

For convenience, we start by recalling the definition of $\infty$-categories of fibrant objects.

\begin{defn}
\cite[Definition 7.5.7]{HCHA}\label{def:oofibrant} An \textit{$\infty$-category of fibrant objects} is an $\infty$-category $\cat C$ equipped with two wide subcategories $\cat W,\cat F\subset\cat C$, whose morphisms are called \textit{weak equivalences} and \textit{fibrations}, satisfying the following conditions:

\begin{enumerate}
\item $\cat C$ has a terminal object $\ast\in\cat C$.
\item $\cat W$ has the two-out-of-three property.
\item For every $x\in\cat C$, the map $x\to\ast$ is a fibration.
\item For any fibration $x\to z$ and every morphism $y\to z$ in $\cat C$, the pullback $x\times_{z}y$ exists.
\item Fibrations and acyclic fibrations (morphisms in $\cat W\cap\cat F$) are stable under pullback.
\item Every morphism in $\cat C$ can be factored as a weak equivalence followed by a fibration.
\end{enumerate}
\end{defn}

\begin{rem}
Functorial factorizations are available in many examples of $\infty$-categories of fibrant objects, but are not part of Definition \ref{def:oofibrant}.
Theorem \ref{thm:fibrationcats} likewise makes no functoriality assumption.
This generality is relevant, for example, to the use of fibration categories in homotopy type theory \cite{AKL15,KS19}.
\end{rem}

Here is the main result of this subsection.

\begin{thm}
\label{thm:fibrationcats}Let $\pr{\cat C,\cat W,\cat F}$ be an $\infty$-category of fibrant objects.
Then $\pr{\cat C,\cat W}$ supports right fractions.
\end{thm}

Our proof of Theorem \ref{thm:fibrationcats} draws heavily on the argument of \cite{BM11}, in which the authors prove an analogous claim for Waldhausen (1-)categories.
First, we need the following lemma.

\begin{lem}
\cite[Lemma 4.3.15]{HCHA}\label{lem:4.3.15} Let $\cat X$ be a nonempty $\infty$-category.
Suppose that, for every nonempty finite poset $P$ and every functor $\phi\from P\to\cat X$, there is a zig-zag of natural transformations from $\phi$ to a constant functor.
Then $\cat X$ is weakly contractible.
\end{lem}

We use this lemma to show that, even though $\infty$-categories of fibrant objects lack functorial factorizations in general, factorizations of maps are ``essentially unique.''
\begin{construction}
Let $\cat C$ be an $\infty$-category of fibrant objects.
We write $\Fact\pr{\cat C}\subset\Zig_{[2]}\pr{\cat C,\cat W}$ for the full subcategory spanned by the objects of the form
\[
x\xrightarrow[\sim]{f}y\stackrel{p}{\epi}z,
\]
where $f$ is a weak equivalence and $p$ is a fibration.
Composition determines a functor $\comp\from\Fact\pr{\cat C}\to\Zig_{[1]}\pr{\cat C, \cat W}.$
\end{construction}
\begin{prop}
	\label{prop:Fact}Let $\cat C$ be an $\infty$-category of fibrant objects. The functor $\comp\from \Fact(\cat{C})\to \Zig_{[1]}(\cat{C},\cat{W})$ is a cartesian fibration with weakly contractible fibers.
\end{prop}

\begin{proof}
For each $x\in \cat{C}$, write $\ps{\Fact(\cat{C})}{x}$ for the fiber of $\ev_0\from \Fact(\cat{C})\to \cat{W}$ over $x$.
We first show that $\comp$ induces a cartesian fibration
	\[
		\ps{\Fact(\cat{C})}{x} \to \ps{\Zig_{[1]}(\cat C, \cat W)}{x}
	\]
with weakly contractible fibers.
We can identify this functor with the restriction of $\ev_1\from \Fun([1],\cat{C}_{x/})\to \cat{C}_{x/}$.
Consequently, this functor is a cartesian fibration, with cartesian lifts given by pullbacks. 
This is well-defined, as the pullback of a weak equivalence along a fibration is a weak equivalence by \cite[Proposition 7.4.16]{HCHA}; the two-out-of-three property then shows that the resulting factorization again belongs to $\ps{\Fact(\cat{C})}{x}$.
To show that its fibers are weakly contractible, fix an arbitrary object $\pr{u\from x\to z}\in \ps{\Zig_{[1]}(\cat C, \cat W)}{x}$.
The fiber of $u$ is the full subcategory $\cat K\subset\pr{\cat C_{x/}}_{/u}\simeq\pr{\cat C_{/z}}_{u/}$ spanned by the objects of the form 
\[\begin{tikzcd}
	& x & \\
	y && z,
	\arrow["f"',"\sim", from=1-2, to=2-1]
	\arrow["u", from=1-2, to=2-3]
	\arrow["p"', two heads, from=2-1, to=2-3]
\end{tikzcd}\]
where $f$ is a weak equivalence and $p$ is a fibration.
We must show that $\cat K$ is weakly contractible. 

By the definition of $\infty$-category of fibrant objects, the $\infty$-category $\cat K$ is nonempty.
Thus, by Lemma \ref{lem:4.3.15}, it will suffice to prove the following: For every finite nonempty poset $P$ and every functor $\phi\from P\to\cat K$, one can find a zig-zag of natural transformations connecting $\phi$ and a constant functor.
For this, regard $\phi$ as a diagram $\overline{\phi}\from P\to\cat C_{/z}$ equipped with a natural transformation $\alpha\from\delta\pr u\to\overline{\phi}$, where $\delta\pr u$ denotes the constant diagram at $u$.
Replacing $\overline{\phi}$ by its Reedy fibrant replacement in $\cat C_{/z}$, we may assume that $\overline{\phi}$ is Reedy fibrant.
(See around \cite[Theorem 7.4.20]{HCHA} for material on Reedy fibrancy.)
The replacement has values fibrant over $z$ and the replacement map is a pointwise weak equivalence.
The maps from $u$ are weak equivalences by definition, and the transition maps of the original diagram are weak equivalences by two-out-of-three in their defining triangles under $u$.
Naturality and two-out-of-three therefore show that the replacement again defines a diagram in $\cat K$, connected to $\phi$ by a pointwise weak natural transformation.
Then by \cite[Proposition 7.4.8]{HCHA}, the limit $\lim\overline{\phi}\in\cat C_{/z}$ exists and is fibrant (i.e., it is a fibration with target $z$).
We then factor the map $u\to\lim\overline{\phi}$ as $u\xrightarrow{g}v\xrightarrow{q}\lim\overline{\phi}$, where the images of $g$ and $q$ in $\cat C$ are weak equivalences and fibrations, respectively.
The induced factorization $\delta\pr u\to\delta\pr v\to\overline{\phi}$ connects $\phi$ to a constant diagram.
The map $v\to z$ is a fibration, and each map $v\to\overline{\phi}\pr p$ is a weak equivalence by two-out-of-three, since both $u\to v$ and $u\to\overline{\phi}\pr p$ are.
Thus $v\in\cat K$, and we have verified the hypothesis of Lemma \ref{lem:4.3.15}, as required.

To complete the proof, note that the evaluation functors $\ev_0\from \Fact(\cat{C})\to \cat{W}$ and $\ev_0\from \Zig_{[1]}(\cat{C},\cat{W}) \to \cat{W}$ are cartesian fibrations, with cartesian transport given by precomposition.
This transport preserves the cartesian morphisms of the fiberwise functors considered above, since those morphisms are obtained by pulling back the final fibration.
It now follows from \cite[Lemma A.1.8]{HMS22} that $\comp$ is a cartesian fibration.
\end{proof}

We can now prove Theorem \ref{thm:fibrationcats}.

\begin{proof}
[Proof of Theorem \ref{thm:fibrationcats}] We must show that, for every pair of objects $x,y\in\cat C$ and every pair of integers $i,j\geq0$, the map
\[
\theta\from\Zig_{[-1,i+j]}\pr{\cat C,\cat W}_{x,y}\to\Zig_{[-1,i,-1,j]}\pr{\cat C,\cat W}_{x,y}
\]
is a weak homotopy equivalence.
If $i=0$, the map $\abs{\theta}$ has a homotopy inverse, given by composing the two consecutive weak equivalences.
We may therefore assume that $i>0$.

Consider the functor $\varphi\from\Zig_{[-1,i,-1,j]}\pr{\cat C,\cat W}_{x,y}\to\Zig_{[1]}\pr{\cat C,\cat W}$ which sends
\begin{equation}\label{eq:Fibration_Category_Zigzag}
x\xleftarrow{\sim}a_0\to\cdots\to a_i\xleftarrow[u]{\sim}b_0\to\cdots\to b_j=y
\end{equation}
to the graph $\pr{\id_{b_0},u}\from b_0\to b_0\times a_i$.
This is well-defined because products preserve weak equivalences between fibrant objects, by \cite[Proposition 7.4.16]{HCHA}.
Define $\widetilde{\Zig}_{[-1,i,-1,j]}\pr{\cat C,\cat W}_{x,y}$ and $\widetilde{\Zig}_{[-1,i+j]}\pr{\cat C,\cat W}_{x,y}$ by the pullback squares
\[\begin{tikzcd}
	{\widetilde{\Zig}_{[-1,i+j]}(\cat C,\cat W)_{x,y}} & {\widetilde{\Zig}_{[-1,i,-1,j]}(\cat C,\cat W)_{x,y}} & {\Fact(\cat C)} \\
	{\Zig_{[-1,i+j]}(\cat C,\cat W)_{x,y}} & {\Zig_{[-1,i,-1,j]}(\cat C,\cat W)_{x,y}} & {\Zig_{[1]}(\cat C,\cat W).}
	\arrow["{\widetilde{\theta}}", from=1-1, to=1-2]
	\arrow["\pi"', from=1-1, to=2-1]
	\arrow["\lrcorner"{description, pos=0.1}, draw=none, from=1-1, to=2-2]
	\arrow[from=1-2, to=1-3]
	\arrow["{\pi'}"', from=1-2, to=2-2]
	\arrow["\lrcorner"{description, pos=0.1}, draw=none, from=1-2, to=2-3]
	\arrow["{\comp}", from=1-3, to=2-3]
	\arrow["\theta"', from=2-1, to=2-2]
	\arrow["\varphi"', from=2-2, to=2-3]
\end{tikzcd}\]
Proposition \ref{prop:Fact} shows that $\pi'$ is a cartesian fibration with weakly contractible fibers, and hence a weak homotopy equivalence; the same is true of its base change $\pi$.
Since weak homotopy equivalences have the two-out-of-six property, it suffices to construct a dashed map that makes both triangles in the following diagram commute:
\[\begin{tikzcd}
	{|\widetilde{\Zig}_{[-1,i]}(\cat{C},\cat{W})_{x,y}|} & {|\widetilde{\Zig}_{[-1,i,-1]}(\cat{C},\cat{W})_{x,y}|} \\
	{|\Zig_{[-1,i]}(\cat{C},\cat{W})_{x,y}|} & {|\Zig_{[-1,i,-1]}(\cat{C},\cat{W})_{x,y}|.}
	\arrow["{|\widetilde{\theta}|}", from=1-1, to=1-2]
	\arrow["{|\pi|}"', "\sim", from=1-1, to=2-1]
	\arrow[dashed, from=1-2, to=2-1]
	\arrow["{|\pi'|}", "\sim"', from=1-2, to=2-2]
	\arrow["{|\theta|}"', from=2-1, to=2-2]
\end{tikzcd}\]

To define this, note that an object of $\widetilde{\Zig}_{[-1,i,-1,j]}\pr{\cat C,\cat W}_{x,y}$ consists of a zig-zag as in \eqref{eq:Fibration_Category_Zigzag}, together with a factorization
\[
b_0\xrightarrow[s]{\sim}z\epi b_0\times a_i
\]
of the graph of $u$.
The two projections $z\to b_0$ and $z\to a_i$ are acyclic fibrations.
Indeed, they are fibrations, the first is a retraction of $s$, and the composite of the second with $s$ is $u$. We then define a functor
\[
\psi\from\widetilde{\Zig}_{[-1,i,-1,j]}\pr{\cat C,\cat W}_{x,y}\to
\Zig_{[-1,i+j]}\pr{\cat C,\cat W}_{x,y}
\]
by sending this object to
\[
x\xleftarrow{\sim}a_0\times_{a_i}z\to\cdots\to
a_{i-1}\times_{a_i}z\to b_0\to\cdots\to b_j=y.
\]
Here the map to $b_0$ is induced by the projection $z\to b_0$.
The pullbacks exist because $z\to a_i$ is a fibration, and their projections to the corresponding $a_k$ are acyclic fibrations.
Moreover, a pointwise weak equivalence between inputs induces a pointwise weak equivalence between the resulting zig-zags, by the two-out-of-three property.
Thus $\psi$ is well-defined.

The following diagram gives a zig-zag of pointwise weak natural transformations between $\pi'$ and $\theta\psi$:
\[\begin{tikzcd}[column sep=small]
	x & a_0 \arrow[l, "\sim"'] \arrow[r] & \cdots \arrow[r] & a_{i-1} \arrow[r] & a_i & b_0 \arrow[l, "u"', "\sim"] \arrow[r] & \dots \\
	x \arrow[u, equal] \arrow[d, equal] & a_0\times_{a_i}z \arrow[l, "\sim"'] \arrow[r] \arrow[u, "\sim"] \arrow[d, equal] & \cdots \arrow[r] \arrow[u, "\sim"] \arrow[d, equal] & a_{i-1}\times_{a_i}z \arrow[r] \arrow[u, "\sim"] \arrow[d, equal] & z \arrow[u, "\sim"] \arrow[d, "\sim"'] & b_0 \arrow[l, "s"', "\sim"] \arrow[u, equal] \arrow[d, equal] \arrow[r] & \dots \\
	x & a_0\times_{a_i}z \arrow[l, "\sim"'] \arrow[r] & \cdots \arrow[r] & a_{i-1}\times_{a_i}z \arrow[r] & b_0 & b_0 \arrow[l, "\id"', "\sim"] \arrow[r] & \dots.
	\arrow["\lrcorner"{anchor=center, pos=0.125, rotate=90}, draw=none, from=2-2, to=1-3]
	\arrow["\lrcorner"{anchor=center, pos=0.125, rotate=90}, draw=none, from=2-3, to=1-4]
	\arrow["\lrcorner"{anchor=center, pos=0.125, rotate=90}, draw=none, from=2-4, to=1-5]
\end{tikzcd}\]
This results in a homotopy $\abs{\pi'}\simeq\abs{\theta}\abs{\psi}$. For the other homotopy, note that the map $u$ is the identity on the image of $\widetilde{\theta}$. The map $s$ is then a common section of the two projections $z\rightrightarrows b_0$.
The induced sections of the pullback projections give the pointwise weak natural transformation displayed below:
\[\begin{tikzcd}[column sep=small]
	x \arrow[d, equal] & a_0 \arrow[l, "\sim"'] \arrow[r] \arrow[d, "\sim"'] & \cdots \arrow[r] \arrow[d, "\sim"'] & a_{i-1} \arrow[r] \arrow[d, "\sim"'] & b_0 \arrow[r] \arrow[d, equal] & \cdots \\
	x & a_0\times_{b_0}z \arrow[l, "\sim"'] \arrow[r] & \cdots \arrow[r] & a_{i-1}\times_{b_0}z \arrow[r] & b_0 \arrow[r] & \cdots.
\end{tikzcd}\]
Thus the diagram defines a pointwise weak natural transformation $\pi\to\psi\widetilde{\theta}$, and it follows that $\abs{\pi}\simeq\abs{\psi}\widetilde{\abs{\theta}}$. Since $\pi$ and $\pi'$ are weak homotopy equivalences, both $\theta\psi$ and $\psi\widetilde{\theta}$ are weak homotopy equivalences, and by the two-out-of-six property the same follows for $\theta$.
\end{proof}

\begin{rem}
The proof of Theorem \ref{thm:fibrationcats} simplifies if $\cat C$ admits a functorial factorization.
In this case, one can factor the graph of $u$ functorially and define $\psi$ and the comparison transformations directly, without passing to the categories $\widetilde{\Zig}$.
In particular, there is no need to appeal to Lemma \ref{lem:4.3.15} in this case.
\end{rem}

\section{\label{sec:DKcomparison}Comparison with Dwyer--Kan's hammock localization}

In this section, we explain the precise sense in which our hammock localization is compatible with Dwyer--Kan's hammock localization (Theorem \ref{thm:comparison}).

\begin{convention}
\label{conv:qcat_ham}Throughout this section, we will use quasicategories as our model of $\infty$-categories.
(This is somewhat forced on us, because the use of simplicial sets is essential in Dwyer--Kan's hammock localization.)
We will not notationally distinguish between categories and their nerves. 

A possible source of confusion is the fact that both animae and quasicategories can be modeled by simplicial sets.
Unless stated otherwise, we regard the category $\sSet$ of simplicial sets as equipped with the Kan--Quillen model structure (which models $\An$) and write $L_{\An}\from\sSet\to\An$ for the localization.
The localization functor for the Joyal model structure (which models $\infty$-categories) is denoted by $L_{\Cat_{\infty}}:\sSet\to\Cat_{\infty}$.
Note that the diagram
\[\begin{tikzcd}
	\sSet & \sSet \\
	{\Cat_\infty} & \An
	\arrow["\id", from=1-1, to=1-2]
	\arrow["{L_{\Cat_\infty}}"', from=1-1, to=2-1]
	\arrow["{L_\An}", from=1-2, to=2-2]
	\arrow["{|-|}"', from=2-1, to=2-2]
\end{tikzcd}\]
commutes up to equivalence.
\end{convention}

\begin{convention}
\label{conv:Zig}Following Convention \ref{conv:qcat_ham}, let $\cat C$ be a quasicategory and let $\cat W\subset\cat C$ be a wide subquasicategory.
For $Z\in\Zig$, we will use the following explicit simplicial-set model of $\Zig_{Z}\pr{\cat C,\cat W}$.
The definition is by induction on the number of arrows in $Z$, by setting $\Zig_{[0]}\pr{\cat C,\cat W}=\cat W$ and 
\begin{align*}
\Zig_{[Z,1]}\pr{\cat C,\cat W} & =\Zig_{Z}\pr{\cat C,\cat W}\times_{\cat W}\cat W\times_{\Fun\pr{\{0\},\cat C}}\Fun\pr{[1],\cat C}\times_{\Fun\pr{\{1\},\cat C}}\cat W,\\
\Zig_{[Z,-1]}\pr{\cat C,\cat W} & =\Zig_{Z}\pr{\cat C,\cat W}\times_{\cat W}\cat W\times_{\Fun\pr{\{1\},\cat W}}\Fun\pr{[1],\cat W}\times_{\Fun\pr{\{0\},\cat W}}\cat W.
\end{align*}
Here, $\Fun$ denotes the internal hom of $\sSet$, and all the limits are formed in $\sSet$.
Since the Joyal model structure is cartesian \cite[Corollary 3.6.4]{JT07}, the evaluation maps $\Fun\pr{[1],\cat C}\to\Fun\pr{\{i\},\cat C}$ are Joyal fibrations, and hence the displayed limits are homotopy limits.
After applying $L_{\Cat_{\infty}}$, this construction agrees with Definition \ref{def:Z_CW}.
\end{convention}

To state the main result of this section, we need a few definitions.
We start by recalling Dwyer--Kan's hammock localization.

\begin{defn}
\cite[2.1 and Proposition 5.5]{DK80_2}\label{def:hammock_DK} Let $\pr{\cat C,\cat W}$ be a relative ($1$-)category.
For each pair of objects $x,y\in\cat C$, we define a simplicial set $L^{H}\pr{\cat C,\cat W}\pr{x,y}$ by
\[
L^{H}\pr{\cat C,\cat W}\pr{x,y}=\colim_{Z\in\Zig^{\op}_{\partial}}\Zig_{Z}\pr{\cat C,\cat W}_{x,y},
\]
where the colimit is taken in the category $\sSet$.
By concatenating zig-zags, we can assemble these simplicial sets into a simplicial category $L^{H}\pr{\cat C,\cat W}$, called the \textbf{hammock localization}.
\end{defn}

To compare $L^{H}\pr{\cat C,\cat W}$ with our construction, we must convert simplicial categories to Segal spaces.
For this, we use the following construction:

\begin{defn}
\label{def:N}We write $\sCat$ for the category of simplicial categories, i.e., categories enriched over $\sSet$.
We define a functor
\[
\mathbb{N}\from\sCat\to\Fun\pr{\Del^{\op},\sSet}
\]
by
\[
\mathbb{N}\pr{\cat C}_{n}=\coprod_{x_{0},\dots,x_{n}\in\cat C}\cat C\pr{x_{0},x_{1}}\times\cdots\times\cat C\pr{x_{n-1},x_{n}}.
\]
Note that the simplicial anima $L_{\An}\mathbb{N}\pr{\cat C}_{\bullet}$ is Segal, because the functor $L_{\An}\from\sSet\to\An$ preserves products. 
\end{defn}

\begin{rem}
Recall that by the work of Joyal--Tierney \cite{JT07} and Lurie \cite[Theorem 2.2.5.1]{HTT}, both $\sCat$ and $\Fun\pr{\Del^{\op},\sSet}$ are models of $\infty$-categories, in that $\Cat_{\infty}$ is a localization of these $1$-categories.
The functor $\mathbb{N}$ induces a self-equivalence $\Cat_{\infty}\xrightarrow{\simeq}\Cat_{\infty}$ when localized at appropriate maps (Theorem \ref{thm:N}).
Since we never need this fact, we only record a proof of this in Appendix \ref{app:scats_to_sspace}.
\end{rem}

Next, we give a concrete model of $\cat Z_{\cat C,\cat W}$.

\begin{defn}
Let $\pr{\cat C,\cat W}$ be a relative category.
We define $\mathbb{Z}_{\cat C,\cat W}\in\Fun\pr{\Del^{\op},\sSet}$ to be the left Kan extension of the functor
\begin{align*}
\pr{\int\Zig^{\bullet}_{\partial}}^{\op} & \to\sSet,\\
\pr{Z_{1},\dots,Z_{n}} & \mapsto\Zig_{Z_{1}\cdots Z_{n}}\pr{\cat C,\cat W}.
\end{align*}
\end{defn}

We can now state the main result of this section.

\begin{thm}
\label{thm:comparison}Let $\pr{\cat C,\cat W}$ be a relative category.
Then:

\begin{enumerate}
\item The composite
\[
\Del^{\op}\xrightarrow{\mathbb{Z}_{\cat C,\cat W}}\sSet\xrightarrow{L_{\An}}\An
\]
is equivalent to $\cat Z_{\cat C,\cat W}$.
\item Concatenation of zig-zags determines a map
\[
\theta\from\mathbb{N}\pr{L^{H}\pr{\cat C,\cat W}}\to\mathbb{Z}_{\cat C,\cat W}.
\]
The image of this map in $\Fun\pr{\Del^{\op},\An}$ is a fully faithful and essentially surjective map of Segal animae.
\end{enumerate}
\end{thm}

The proof of Theorem \ref{thm:comparison} relies crucially on the following recognition results for $\infty$-categorical colimits as $1$-categorical colimits. 
\begin{prop}
\label{prop:infinity_1_colimit}Let $\pr{\cat C,\cat W}$ be a relative category.
The functor $L_{\An}\from\sSet\to\An$ preserves the following colimits:

\begin{enumerate}[label=(\alph*)]

	\item For every $n\geq0$, the colimit $\colim_{\pr{Z_{1},\dots,Z_{n}}\in(\Zig^{n}_{\partial})^{\op}}\Zig_{Z_{1}\cdots Z_{n}}\pr{\cat C,\cat W}$.

	\item For every pair of objects $x,y\in\cat C$, the colimit $\colim_{Z\in\pr{\Zig_{\partial}}^{\op}}\Zig_{Z}\pr{\cat C,\cat W}_{x,y}$.

\end{enumerate}
\end{prop}

Assuming Proposition \ref{prop:infinity_1_colimit} for now, we can prove Theorem \ref{thm:comparison} as follows:

\begin{proof}
[Proof of Theorem \ref{thm:comparison}]
Part (1) follows from part (a) of Proposition \ref{prop:infinity_1_colimit}. 

For part (2), we first construct the map $\theta$.
By definition, we have
\[
\hspace{-26pt}
\mathbb{N}\pr{L^{H}\pr{\cat C,\cat W}}_{n}=\coprod_{x_{0},\dots,x_{n}\in\cat C}\colim_{\pr{Z_{1},\dots,Z_{n}}\in\pr{\Zig^{n}_{\partial}}^{\op}}\Zig_{Z_{1}}\pr{\cat C,\cat W}_{x_{0},x_{1}}\times\dots\times\Zig_{Z_{n}}\pr{\cat C,\cat W}_{x_{n-1},x_{n}}
\hspace{-26pt}
\]
and
\[
\pr{\mathbb{Z}_{\cat C,\cat W}}_{n}=\colim_{\pr{Z_{1},\dots,Z_{n}}\in\pr{\Zig^{n}_{\partial}}^{\op}}\Zig_{Z_{1}\cdots Z_{n}}\pr{\cat C,\cat W}.
\]
Therefore, the inclusions
\[
\Zig_{Z_{1}}\pr{\cat C,\cat W}_{x_{0},x_{1}}\times\dots\times\Zig_{Z_{n}}\pr{\cat C,\cat W}_{x_{n-1},x_{n}}\hookrightarrow\Zig_{Z_{1}\cdots Z_{n}}\pr{\cat C,\cat W}
\]
determine a map of simplicial sets $\theta_{n}\from\mathbb{N}\pr{L^{H}\pr{\cat C,\cat W}}_{n}\to\pr{\mathbb{Z}_{\cat C,\cat W}}_{n}$.
The maps $\{\theta_{n}\}_{n\geq0}$ are natural in $[n]\in\Del^{\op}$, so we obtain a map of bisimplicial sets 
\[
\theta\from\mathbb{N}\pr{L^{H}\pr{\cat C,\cat W}}\to\mathbb{Z}_{\cat C,\cat W}.
\]

To complete the proof, we must show that $\theta$ induces a fully faithful and essentially surjective map of simplicial animae.
Essential surjectivity is clear, because the map $\theta$ in the $0$th degree is given by $\ob\cat W\hookrightarrow\cat W$.
Full faithfulness follows from part (2) of Theorem \ref{thm:Z_Segal} and part (b) of Proposition \ref{prop:infinity_1_colimit}.
The proof is now complete.
\end{proof}

\begin{rem}
When $\pr{\cat C,\cat W}$ is the underlying relative category of a model category, Dwyer--Kan moreover identify the mapping simplicial sets of $L^{H}\pr{\cat C,\cat W}$ with the homotopy function complexes obtained from cosimplicial and simplicial resolutions \cite[Propositions 4.4 and 4.5]{DK80_3}.
For a simplicial model category, the mapping simplicial set $L^{H}\pr{\cat C,\cat W}\pr{X,Y}$ from a cofibrant object $X$ to a fibrant object $Y$ is weakly equivalent to the simplicial mapping object from $X$ to $Y$ \cite[Corollary 4.7]{DK80_3}.
\end{rem}

The remainder of this section is devoted to the proof of Proposition \ref{prop:infinity_1_colimit}.
Our main tool will be a formalism of Reedy categories, which we now recall.

\begin{defn}
Given a model category $\mathbf{M}$, we write $\mathbf{M}_{\infty}$ for its localization at weak equivalences.
This is the \textit{underlying $\infty$-category} of $\mathbf{M}$.
\end{defn}

\begin{defn}
\cite[Definition 15.1.2]{Hirschhorn} A \textit{Reedy category} is a category $\cat R$ equipped with a function $\deg\from\ob\cat R\to\mathbb{Z}_{\geq0}$ and two wide subcategories $\cat R^{+},\cat R^{-}\subset\cat R$, satisfying the following three conditions:

\begin{enumerate}
\item Every non-identity morphism in $\cat R^{+}$ raises degree.
\item Every non-identity morphism in $\cat R^{-}$ decreases degree.
\item Every morphism $f$ in $\cat R$ can be written uniquely as $f=f^{+}f^{-}$, where $f^{+}$ and $f^{-}$ are morphisms of $\cat R^{+}$ and $\cat R^{-}$, respectively.
\end{enumerate}
For each object $x\in\cat R$, we write $\partial\cat R^{-}_{x/}\subset\cat R^{-}_{x/}$ for the full subcategory spanned by the objects other than $\id_{x}$.
We define $\partial\cat R^{+}_{/x}$ similarly.
\end{defn}

\begin{example}
The category $\Zig$ becomes a Reedy category with respect to the following data: The degree of an object $[S,T]\in\Zig$ is $\abs S+\abs T$.
A morphism of $\Zig$ belongs to $\Zig^{+}$ if it induces an injection between the sets of objects, and to $\Zig^{-}$ if it induces a surjection between the sets of objects.
Indeed, every morphism factors uniquely as the quotient onto its image on objects followed by the inclusion of that image.
Note that this Reedy structure restricts to Reedy structures on $\Del$ and $\Zig_{\partial}$.
\end{example}

\begin{example}
If $\cat R$ and $\cat S$ are Reedy categories, then $\cat R\times\cat S$ also becomes a Reedy category by setting $\deg\pr{r,s}=\deg r+\deg s$, $\pr{\cat R\times\cat S}^{+}=\cat R^{+}\times\cat S^{+}$, and $\pr{\cat R\times\cat S}^{-}=\cat R^{-}\times\cat S^{-}$.
Unless stated otherwise, we always interpret finite products of Reedy categories to be equipped with this Reedy structure.
\end{example}

Given a Reedy category $\cat R$ and a model category $\mathbf{M}$, the diagram category $\Fun\pr{\cat R,\mathbf{M}}$ admits a model structure, called the \textit{Reedy model structure} \cite[Theorem 15.3.4]{Hirschhorn}.
Explicitly, let $\alpha\from F\to G$ be a morphism in $\Fun\pr{\cat R,\mathbf{M}}$.
Then:

\begin{itemize}
\item $\alpha$ is a weak equivalence if and only if the map $Fx\to Gx$ is a weak equivalence for every $x\in\cat R$.
\item $\alpha$ is a cofibration if and only if the map $L_{x}G\amalg_{L_{x}F}Fx\to Gx$ is a cofibration for each $x\in\cat R$.
Here $L_{x}F=\colim_{y\in\partial\cat R^{+}_{/x}}Fy$ denotes the \textit{latching object} of $F$ at $x$.
\item $\alpha$ is a fibration if and only if the map $Fx\to Gx\times_{M_{x}G}M_{x}F$ is a fibration for each $x\in\cat R$.
Here $M_{x}F=\lim_{y\in\partial\cat R^{-}_{x/}}Fy$ denotes the \textit{matching object} of $F$ at $x$.
\end{itemize}

The functor
\[
\Fun\pr{\cat R,\mathbf{M}}_{\infty}\to\Fun\pr{\cat R,\mathbf{M}_{\infty}}
\]
is an equivalence \cite[Theorem 7.9.8]{HCHA}.
Thus, we have a $1$-categorical model of $\Fun\pr{\cat R,\mathbf{M}_{\infty}}$.
We will use this fact and the theory of derived adjunction \cite[Theorem 7.5.30]{HCHA} to compute $\infty$-categorical colimits as $1$-categorical colimits.
For this, we need the following additional definition on Reedy categories:

\begin{defn}
We say that a Reedy category $\cat R$ has \textit{fibrant constants} if for every (cocomplete) model category $\mathbf{M}$, the colimit functor
\[
\colim\from\Fun\pr{\cat R,\mathbf{M}}\to\mathbf{M}
\]
is left Quillen.
This is equivalent to the condition that, for every object $x\in\cat R$, the category $\partial\cat R^{-}_{x/}$ be either connected or empty.
(See \cite[Proposition 15.10.2 and Theorem 15.10.8]{Hirschhorn}.)

The dual notion of having \textit{cofibrant constants} is defined similarly.
\end{defn}

The following lemmas record closure properties and examples for these conditions.

\begin{lem}
\label{lem:finprod_fibconst}Let $\cat R$ and $\cat S$ be Reedy categories.
If $\cat R$ and $\cat S$ have fibrant constants, then so does $\cat R\times\cat S$.
\end{lem}

\begin{proof}
For every $\pr{x,y}\in\cat R\times\cat S$, we have 
\[
\partial\pr{\cat R\times\cat S}^{-}_{\pr{x,y}/}=\pr{\pr{\partial\cat R^{-}_{x/}}\times\cat S^{-}_{y/}}\cup\pr{\cat R^{-}_{x/}\times\pr{\partial\cat S^{-}_{y/}}}.
\]
The right-hand side is empty if $\partial\cat R^{-}_{x/}=\partial\cat S^{-}_{y/}=\emptyset$, and is connected otherwise since $\cat S^{-}_{y/}$ and $\cat R^{-}_{x/}$ are connected.
\end{proof}

\begin{lem}
\label{lem:Zigpartial_cof_const}For every $n\geq0$, the Reedy category $\Zig^{n}_{\partial}$ has cofibrant constants.
\end{lem}

\begin{proof}
By the dual of Lemma \ref{lem:finprod_fibconst}, it suffices to prove the claim for $n=1$.
In other words, it suffices to show that for every $Z\in\Zig_{\partial}$, the poset $\partial\pr{\Zig_{\partial}}^{+}_{/Z}$ is either empty or connected.
Let $Z^{\mathrm{red}}\subset Z$ be the full subcategory spanned by the boundary vertices and the vertices at which the direction of the arrows changes.
Thus, each maximal string of arrows in $Z$ pointing in the same direction is replaced by its composite, and the inclusion $Z^{\mathrm{red}}\hookrightarrow Z$ is a morphism in $\Zig_{\partial}^{+}$.
Every injective morphism $Y\to Z$ in $\Zig_{\partial}$ contains all vertices of $Z^{\mathrm{red}}$ in its image: An omitted direction-change vertex would lie between two consecutive vertices in the image, but there is no morphism in $Z$ between these two vertices in either direction.
Consequently, if $Z^{\mathrm{red}}\neq Z$, then $Z^{\mathrm{red}}\hookrightarrow Z$ is an initial object of $\partial\pr{\Zig_{\partial}}^{+}_{/Z}$.
If $Z^{\mathrm{red}}=Z$, then every injective boundary-preserving morphism to $Z$ is surjective and hence an identity, so $\partial\pr{\Zig_{\partial}}^{+}_{/Z}$ is empty.
\end{proof}

We now arrive at the proof of Proposition \ref{prop:infinity_1_colimit}.

\begin{proof}
[Proof of Proposition \ref{prop:infinity_1_colimit}]
By Lemma \ref{lem:Zigpartial_cof_const}, it suffices to show that the functors
\begin{align*}
\pr{\Zig^{n}_{\partial}}^{\op} & \to\sSet,\,\pr{Z_{1},\dots,Z_{n}}\mapsto\Zig_{Z_{1}\cdots Z_{n}}\pr{\cat C,\cat W},\\
\Zig^{\op}_{\partial} & \to\sSet,\,Z\mapsto\Zig_{Z}\pr{\cat C,\cat W}_{x,y},
\end{align*}
are Reedy cofibrant.
For the first one, note that its latching object at $\pr{Z_{1},\dots,Z_{n}}\in\Zig^{n}_{\partial}$ is simply the union of the images of the maps $\Zig_{Z'_{1}\cdots Z'_{n}}\pr{\cat C,\cat W}\to\Zig_{Z_{1}\cdots Z_{n}}\pr{\cat C,\cat W}$, where $\pr{Z_{i}\to Z'_{i}}^{n}_{i=1}$ ranges over all objects of $\partial\pr{\Zig^{n}_{\partial}}^{-}_{\pr{Z_{1},\dots,Z_{n}}/}$.
Indeed, each of these maps is an inclusion of simplicial sets, whose image consists of the diagrams in which the arrows collapsed by the quotient maps are identities, and intersections of such images correspond to the common quotient generated by the quotient maps.
The latching map is therefore a monomorphism, hence a cofibration in the Kan--Quillen model structure.
The same argument applies to the second functor, since the quotient maps preserve the endpoints.
\end{proof}

\begin{rem}
With Convention \ref{conv:Zig}, Definition \ref{def:hammock_DK} makes sense even when $\pr{\cat C,\cat W}$ is a relative quasicategory.
In other words, we can associate to each relative $\infty$-category a simplicial category $L^{H}\pr{\cat C,\cat W}$.
Everything in this section remains true under the relaxed assumption that $\pr{\cat C,\cat W}$ is a relative quasicategory.
Thus, the simplicial category $L^{H}\pr{\cat C,\cat W}$ models the localization of $\cat C$ at $\cat W$.
\end{rem}

\begin{rem}
Dwyer--Kan's hammock localization has a variant, whose hom-simplicial sets are nerves of categories.
Explicitly, for a relative category $\pr{\cat C,\cat W}$, the hom-simplicial-set from $x$ to $y$ is the nerve of the Grothendieck construction of the functor $\Zig^{\op}_{\partial}\to\Cat$, $Z\mapsto\Zig_{Z}\pr{\cat C,\cat W}_{x,y}$.
(See \cite[35.6]{DHKS04}.)
We can also generalize this to a construction associating to relative quasicategories a simplicial category, using Lurie's relative nerve construction \cite[3.2.5]{HTT}.
The resulting simplicial category has the advantage that its hom-simplicial sets are quasicategories.
Details are left to the reader.
\end{rem}

\section{\label{sec:MGcomparison}Comparison with Mazel-Gee's hammock localization}

In \cite{MG18}, Mazel-Gee constructed another model of hammock localization, using bisimplicial animae.
In this section, we briefly recall his construction and explain its relation to our simplicial anima of zig-zags.
The main result of this section says that the simplicial anima obtained from Mazel-Gee's hammock localization by taking colimits levelwise agrees, up to a fully faithful and essentially surjective map of Segal animae, with our construction $\cat Z_{\cat C,\cat W}$ (Proposition \ref{prop:MGcomparison}).
Together with Corollary \ref{cor:hammock_anima_is_anima_of_loc}, this ties up a loose end left in \cite{MG18}, namely that the $\infty$-category associated to Mazel-Gee's hammock localization is equivalent to $\cat C[\cat W^{-1}]$ (Corollary \ref{cor:MGhamloc_is_localization}).

We start with a review of Mazel-Gee's construction.
The key idea behind Mazel-Gee's construction is to use $\infty$-categories enriched over simplicial animae, just like Dwyer--Kan used categories enriched over simplicial sets.
For his setting, such an enrichment is provided by bisimplicial animae:

\begin{defn}
\cite[Definitions 2.6, 2.12]{MG18}\label{def:sAn} Write $\sAn=\Fun\pr{\Del^{\op},\An}$ for the $\infty$-category of simplicial animae.
A \textit{Segal simplicial anima} is a simplicial object $X\in\Fun\pr{\Del^{\op},\sAn}$ satisfying the Segal condition.
A Segal simplicial anima $X$ is called a \textit{$\sAn$-enriched $\infty$-category} if further $X_{0}\in\sAn$ is a constant simplicial object.

We write $\Cat_{\sAn}\subset\Seg\pr{\sAn}\subset\Fun\pr{\Del^{\op},\sAn}$ for the full subcategories of $\sAn$-enriched $\infty$-categories and Segal simplicial animae, respectively.
\end{defn}

\begin{rem}
\cite[Definition 2.18, Remark 2.17]{MG18} The inclusion $\Cat_{\sAn}\hookrightarrow\Seg\pr{\sAn}$ admits a right adjoint, which Mazel-Gee denotes by $\sp$.
Explicitly, it is given by the pullback
\[\begin{tikzcd}
	{\sp(X)} & X \\
	{\cosk_0(\delta(X_{0,0}))} & {\cosk_0(X),}
	\arrow[from=1-1, to=1-2]
	\arrow[from=1-1, to=2-1]
	\arrow["\lrcorner"{description, pos=0}, draw=none, from=1-1, to=2-2]
	\arrow[from=1-2, to=2-2]
	\arrow[from=2-1, to=2-2]
\end{tikzcd}\]
where $\cosk_{0}\pr X=X^{\bullet+1}_{0}$ and $\delta\pr{X_{0,0}}\in\sAn$ denotes the constant simplicial object at $X_{0,0}$.
The counit $\sp\pr X\to X$ does not change the mapping simplicial animae.
\end{rem}

Postcomposing with the colimit functor $\colim_{\Del^{\op}}\from\sAn\to\An$ gives a functor
\[
\pr{\colim_{\Del^{\op}}}_{\ast}\from\Fun\pr{\Del^{\op},\sAn}\to\sAn.
\]
While this functor does not preserve Segal objects in general, it does preserve some Segal objects:

\begin{prop}
\cite[Lemma 2.13]{MG18} For every $\sAn$-enriched $\infty$-category $X$, the simplicial anima $\pr{\colim_{\Del^{\op}}}_{\ast}\pr X$ is Segal.
\end{prop}

\begin{defn}
Let $X$ be an $\sAn$-enriched $\infty$-category.
We call $\ac\circ\pr{\colim_{\Del^{\op}}}_{\ast}\pr X$ the \textit{$\infty$-category associated to $X$}.
\end{defn}

He then defines the $\infty$-category of (pre-)hammock localization as follows:

\begin{defn}
\cite[Construction 5.2, Definition 5.4]{MG18} The \textit{pre-hammock localization} functor $\mathscr{L}^{H}_{\pre}\from\RelCat_{\infty}\to\Fun\pr{\Del^{\op},\sAn}$ is defined by 
\[
\mathscr{L}^{H}_{\pre}\pr{\cat C,\cat W}_{n}=\colim_{\pr{Z_{1},\dots,Z_{n}}\in\pr{\Zig^{n}_{\partial}}^{\op}}\N\pr{\Zig_{Z_{1}\cdots Z_{n}}\pr{\cat C,\cat W}},
\]
where $\N\from\Cat_{\infty}\to\sAn$ denotes the Rezk nerve.
(The simplicial structure of $\mathscr{L}^{H}_{\pre}\pr{\cat C,\cat W}$ is as discussed in Definition \ref{def:Z_CW}.) 

The functor $\mathscr{L}^{H}_{\pre}$ actually takes values in $\Seg\pr{\sAn}$ \cite[Lemma 5.3]{MG18}.
With this in mind, we define the \textit{hammock localization} functor $\mathscr{L}^{H}$ by
\[
\mathscr{L}^{H}\pr{\cat C,\cat W}=\sp\circ\mathscr{L}^{H}_{\pre}\pr{\cat C,\cat W}.
\]
\end{defn}

We are now ready to compare Mazel-Gee's model of hammock localization with our simplicial anima $\cat Z_{\cat C,\cat W}$.

\begin{prop}
\label{prop:MGcomparison}Let $\pr{\cat C,\cat W}$ be a relative $\infty$-category. 
\begin{enumerate}
\item There is a natural equivalence
\[
\pr{\colim_{\Del^{\op}}}_{\ast}\circ\mathscr{L}^{H}_{\mathrm{pre}}\pr{\cat C,\cat W}\simeq\cat Z_{\cat C,\cat W}.
\]
\item The map
\[
\pr{\colim_{\Del^{\op}}}_{\ast}\circ\mathscr{L}^{H}\pr{\cat C,\cat W}\to\pr{\colim_{\Del^{\op}}}_{\ast}\circ\mathscr{L}^{H}_{\mathrm{pre}}\pr{\cat C,\cat W}\simeq\cat Z_{\cat C,\cat W}
\]
is a fully faithful and essentially surjective map of Segal animae.
\end{enumerate}
\end{prop}

\begin{proof}
For (1), it suffices to show that the composite
\[
\Cat_{\infty}\xrightarrow{\N}\sAn\xrightarrow{\colim_{\Del^{\op}}}\An
\]
is equivalent to $\abs -\from\Cat_{\infty}\to\An$.
To see this, let $\ac\from\sAn\to\Cat_{\infty}$ denote the left adjoint of $\N$.
Then $\colim_{\Del^{\op}}\simeq\abs -\circ\ac$ because they have equivalent right adjoints, and $\ac\circ\N\simeq\id$ because $\N$ is fully faithful.
Thus $\colim_{\Del^{\op}}\circ\N\simeq\abs -\circ\ac\circ\N\simeq\abs -$, as required.

Next, we turn to (2).
For essential surjectivity, note that the map $\mathscr{L}^{H}\pr{\cat C,\cat W}_{0}\to\mathscr{L}^{H}_{\mathrm{pre}}\pr{\cat C,\cat W}_{0}$ can be identified with $\N\pr{\cat W^{\simeq}}\to\N\pr{\cat W}$.
Upon passing to colimits, we can identify this map with $\abs{\cat W^{\simeq}}\to\abs{\cat W}$, which is essentially surjective.
To prove full faithfulness, we first compute the mapping animae of the left-hand side. By the preceding remark, passage from $\mathscr{L}^{H}_{\mathrm{pre}}$ to $\mathscr{L}^{H}$ does not change the mapping simplicial animae. We therefore have a fiber sequence
\[
\colim_{Z\in\Zig^{\op}_{\partial}}\N\pr{\Zig_{Z}\pr{\cat C,\cat W}_{x,y}}\to\mathscr{L}^{H}\pr{\cat C,\cat W}_{1}\to\mathscr{L}^{H}\pr{\cat C,\cat W}_{0}\times\mathscr{L}^{H}\pr{\cat C,\cat W}_{0}
\]
for each $x,y\in\cat C$.
Since $\mathscr{L}^{H}\pr{\cat C,\cat W}_{0}\simeq\delta\pr{\cat W^{\simeq}}$ is constant, the source-target map exhibits $\mathscr{L}^{H}\pr{\cat C,\cat W}_{1}$ as a simplicial object over the fixed anima $\cat W^{\simeq}\times\cat W^{\simeq}$.
Pullback along $\pr{x,y}\colon\ast\to\cat W^{\simeq}\times\cat W^{\simeq}$ preserves colimits, so realization commutes with the fiber over $\pr{x,y}$.
Consequently, full faithfulness amounts to the map
\[
\colim_{Z\in\Zig_{\partial}^{\op}}\abs{\Zig_{Z}\pr{\cat C,\cat W}_{x,y}}\to\pr{\colim_{Z\in\Zig_{\partial}^{\op}}\abs{\Zig_{Z}\pr{\cat C,\cat W}}}_{x,y}
\]
being an equivalence.
We already proved this in Theorem \ref{thm:Z_Segal}.
\end{proof}

\begin{cor}
\label{cor:MGhamloc_is_localization}Let $\pr{\cat C,\cat W}$ be a relative $\infty$-category.
The $\infty$-category associated to Mazel-Gee's hammock localization $\mathscr{L}^{H}\pr{\cat C,\cat W}$ is equivalent to $\cat C[\cat W^{-1}]$.
\end{cor}

\begin{proof}
This follows from Proposition \ref{prop:MGcomparison} and Corollary \ref{cor:hammock_anima_is_anima_of_loc}.
\end{proof}

\appendix

\section{\label{app:abstract_nonsense}Abstract nonsense}

In this appendix, we collect several standard results about $\infty$-categories that we need in the paper.

\subsection{Quillen's Theorem B}
\begin{prop}
\cite[Lemma A.1.1]{KSW26}, \cite[Proposition A.3]{Ara26}\label{prop:ThmB} Let $p\from\cat A\to\cat B$ be a cartesian fibration of $\infty$-categories.
Suppose that, for every morphism $b\to b'$ in $\cat B$, the induced map $\abs{\cat A_{b'}}\to\abs{\cat A_{b}}$ is an equivalence.
Then every pullback square of $\infty$-categories 
\[\begin{tikzcd}
	{\mathcal{A}'} & {\mathcal{A}} \\
	{\mathcal{B}'} & {\mathcal{B}}
	\arrow[from=1-1, to=1-2]
	\arrow[from=1-1, to=2-1]
	\arrow[from=1-2, to=2-2]
	\arrow[from=2-1, to=2-2]
\end{tikzcd}\]
remains a pullback after applying $\abs -\from\Cat_{\infty}\to\An$.
\end{prop}

\begin{cor}
\label{cor:ThmB}Let $p\from\cat A\to\cat B$ be a functor of $\infty$-categories.
Suppose that, for every map $b\to b'$ in $\cat B$, the induced map
\[
\abs{\cat A_{b'/}}\to\abs{\cat A_{b/}}
\]
is an equivalence.
Then for every cocartesian fibration $q\from\cat E\to\cat B$, the square 
\[\begin{tikzcd}
	{|\mathcal{E}\times_{\mathcal{B}}\mathcal{A}|} & {|\mathcal{A}|} \\
	{|\mathcal{E}|} & {|\mathcal{B}|}
	\arrow[from=1-1, to=1-2]
	\arrow[from=1-1, to=2-1]
	\arrow[from=1-2, to=2-2]
	\arrow[from=2-1, to=2-2]
\end{tikzcd}\]
is a pullback.
\end{cor}

\begin{proof}
Let $p'\from\cat A'\to\cat B$ denote the free right fibration generated by $p$, so that $p$ factors as $\cat A\xrightarrow{i}\cat A'\xrightarrow{p'}\cat B$.
Explicitly, $p'$ is the fiberwise groupoidification of the free cartesian fibration, which is given by the projection $\cat B\comma p\to\cat B$ of the comma $\infty$-category of $p$ \cite[Theorem 4.5]{GHN17}.
Note that $i$ is final, because the map $\cat A\to\cat B\comma p$ is final (being a right adjoint) and the map $\cat B\comma p\to\cat A'$ is final (being a localization \cite[\href{https://kerodon.net/tag/02N9}{Tag 02N9}]{kerodon}).

We now consider the following diagram:
\[\begin{tikzcd}
	{|\mathcal{E}\times_{\mathcal{B}}\mathcal{A}|} & {|\mathcal{A}|} \\
	{|\mathcal{E}\times_{\mathcal{B}}\mathcal{A}'|} & {|\mathcal{A}'|} \\
	{|\mathcal{E}|} & {|\mathcal{B}|.}
	\arrow[from=1-1, to=1-2]
	\arrow["{|q^*i|}", from=1-1, to=2-1]
	\arrow["{|i|}", from=1-2, to=2-2]
	\arrow[from=2-1, to=2-2]
	\arrow[from=2-1, to=3-1]
	\arrow[from=2-2, to=3-2]
	\arrow["{|q|}"', from=3-1, to=3-2]
\end{tikzcd}\]
Since $q$ is a cocartesian fibration, it is smooth \cite[Proposition 4.1.2.15]{HTT}.
Thus, $q^{*}i$ is final.
It follows that $\abs i$ and $\abs{q^{*}i}$ are equivalences.
The fiber of $p'$ over $b\in\cat B$ is $\abs{\cat A_{b/}}$, so the hypothesis implies that transport in $p'$ induces equivalences between its fibers.
Proposition \ref{prop:ThmB} therefore shows that the bottom square is cartesian.
Thus, the outer rectangle is cartesian, as desired.
\end{proof}

\subsection{Weighted colimits}

\begin{lem}
\label{lem:two-sided-groth}Let $\cat X\xrightarrow{p}\cat B\xleftarrow{q}\cat Y$ be functors of $\infty$-categories.
Suppose that $p$ is a cocartesian fibration and $q$ is a cartesian fibration, corresponding to $X\from\cat B\to\Cat_{\infty}$ and $Y\from\cat B^{\op}\to\Cat_{\infty}$, respectively.
Then we have
\[
\abs{\cat X\times_{\cat B}\cat Y}\simeq\int^{b\in\cat B}\abs{X\pr b}\times\abs{Y\pr b},
\]
where the right-hand side denotes a coend \cite{Hau22}.
\end{lem}

\begin{proof}
Fiberwise groupoid completion gives localizations of the total categories \cite[\href{https://kerodon.net/tag/02LW}{Tag 02LW}]{kerodon}, which are both initial and final \cite[\href{https://kerodon.net/tag/02N9}{Tag 02N9}]{kerodon}.
Their pullbacks along the other fibration remain initial or final by smoothness of cocartesian fibrations and properness of cartesian fibrations.
Consequently, these replacements do not change the realization of the pullback, and we may assume that $p$ and $q$ are left and right fibrations, respectively.
By \cite[Corollary 3.3.4.6]{HTT}, the left-hand side can be identified with $\colim\pr{\cat Y\xrightarrow{q}\cat B\xrightarrow{X}\An}$, and this is equivalent to the right-hand side by \cite[Proposition 4.3]{Hau22}.
\end{proof}

\begin{cor}
\label{cor:fiberwise_pullback}Let $u\from\cat X\to\cat X'$ and $v\from\cat Y\to\cat Y'$ be functors over an $\infty$-category $\cat B$ which induce equivalences on the realizations of all fibers.
Suppose that $\cat X$ and $\cat X'$ are both cartesian or both cocartesian fibrations over $\cat B$, and likewise for $\cat Y$ and $\cat Y'$.
Then the induced map
\[
\abs{\cat X\times_{\cat B}\cat Y}\to\abs{\cat X'\times_{\cat B}\cat Y'}
\]
is an equivalence.
\end{cor}

\begin{proof}
By the argument in the proof of Lemma \ref{lem:two-sided-groth}, fiberwise localization does not change the realizations of the pullbacks in question.
We may therefore assume that the four fibrations are left or right fibrations; the functors $u$ and $v$ then automatically preserve cocartesian or cartesian morphisms.
If the two pairs have opposite variance, the claim follows from Lemma \ref{lem:two-sided-groth}. If they have the same variance, the two pullbacks are the corresponding unstraightenings of the pointwise products of the fiber diagrams, so the claim follows by taking colimits after realization.
\end{proof}

\subsection{Universality of the relative Rezk nerve}

The following theorem is due to Mazel-Gee \cite{MG19}.
See also \cite{AC26} for a short proof, and \cite{Ara23} for another short proof and a generalization in the quasicategorical setting.

\begin{thm}
\label{thm:univ_Rezk}Let $\pr{\cat C,\cat W}$ be a relative $\infty$-category.
The map $\N^{\rel}\pr{\cat C,\cat W}\to\N\pr{\cat C[\cat W^{-1}]}$ induces an equivalence
\[
\ac\circ\N^{\rel}\pr{\cat C,\cat W}\xrightarrow{\simeq}\ac\circ\N\pr{\cat C[\cat W^{-1}]}\simeq\cat C[\cat W^{-1}].
\]
\end{thm}

\section{\label{app:scats_to_sspace}From simplicial categories to simplicial spaces}

In this appendix, we record a proof that the functor $\mathbb{N}$ of Definition \ref{def:N} induces an equivalence $\Cat_{\infty}\xrightarrow{\simeq}\Cat_{\infty}$. 
\begin{thm}
\label{thm:N}The functor
\[
\mathbb{N}\from\sCat\to\Fun\pr{\Del^{\op},\sSet}
\]
induces an equivalence of $\infty$-categories when localizing $\sCat$ and $\Fun\pr{\Del^{\op},\sSet}$ at the weak equivalences for the Bergner model structure and the complete Segal space model structure, respectively.
\end{thm}

\begin{proof}
For each simplicial category $\cat C$ and $n\geq0$, define an ordinary category $\cat C_{n}$ as follows: Its objects are the objects of $\cat C$, and its morphisms $x\to y$ are the $n$-simplices of $\cat C\pr{x,y}$.
We write $\Phi\from\sCat\to\Fun\pr{\Del^{\op},\sSet}$ for the functor sending $\cat C$ to the simplicial object whose $n$th object is the nerve of $\cat C_{n}$. 

We then consider the following diagram:
\[\begin{tikzcd}
	\sCat & {\Fun(\Del^{\op},\sSet)} && {\Fun(\Del^{\op},\Cat_{\infty})} \\
	& {\Fun(\Del^{\op},\sSet)} & {\Fun(\Del^{\op},\An)} & {\Cat_{\infty},}
	\arrow["\Phi", from=1-1, to=1-2]
	\arrow["{\mathbb{N}}"', from=1-1, to=2-2]
	\arrow["{\Fun(\Del^{\op},L_{\Cat_\infty})}", from=1-2, to=1-4]
	\arrow["\sigma", from=1-2, to=2-2]
	\arrow["\colim", from=1-4, to=2-4]
	\arrow["{\Fun(\Del^{\op},L_{\An})}"'{yshift=-2pt}, from=2-2, to=2-3]
	\arrow["{\mathrm{ac}}"', from=2-3, to=2-4]
\end{tikzcd}\]
where $\sigma$ is the functor that switches the simplicial coordinates.
The left triangle commutes by definition, and the right rectangle commutes by \cite[Theorems 4.1, 4.5, and 4.11]{JT07}.
Consequently, it will suffice to show that the composite
\[
\colim\circ\Fun\pr{\Del^{\op},L_{\Cat_{\infty}}}\circ\Phi\from\sCat\to\Cat_{\infty}
\]
is a localization.
This follows from \cite[Proposition 1.3.4.14]{HA}, which says that $\cat C\simeq\hocolim_{[n]\in\Del^{\op}}\cat C_{n}$ (combined with the Quillen equivalence between the Bergner and Joyal model structures \cite[Theorem 2.2.5.1]{HTT}).
\end{proof}

\subsection*{Acknowledgment}

K.A. was supported by JSPS KAKENHI Grant Number 24KJ1443.
B.C. is an associate member of the CRC 1785 Generalised Motivic Methods in Geometry.

\bibliographystyle{amsalpha}
\bibliography{refs}

\end{document}

%% file: preamble.tex
\pdfoutput=1
\usepackage{lmodern}
\usepackage[T1]{fontenc}
\usepackage{amsthm}
\usepackage{amssymb}
\usepackage[unicode=true,pdfusetitle,
 bookmarks=true,bookmarksnumbered=false,bookmarksopen=false,
 breaklinks=false,pdfborder={0 0 0},pdfborderstyle={},backref=false,colorlinks=false]
 {hyperref}
\theoremstyle{plain}
\newtheorem{thm}{Theorem}[section]
\theoremstyle{definition}
\newtheorem{defn}[thm]{Definition}
\theoremstyle{definition}
\newtheorem{convention}[thm]{Convention}
\theoremstyle{remark}
\newtheorem{rem}[thm]{Remark}
\theoremstyle{definition}
\newtheorem{notation}[thm]{Notation}
\theoremstyle{definition}
\newtheorem{example}[thm]{Example}
\theoremstyle{definition}
\newtheorem{construction}[thm]{Construction}
\theoremstyle{plain}
\newtheorem{prop}[thm]{Proposition}
\theoremstyle{plain}

\theoremstyle{plain}
\newtheorem{cor}[thm]{Corollary}
\theoremstyle{plain}
\newtheorem{lem}[thm]{Lemma}
\theoremstyle{plain}

\theoremstyle{plain}
\newtheorem{customthm}{Theorem}

\newtheorem{customcor}[customthm]{Corollary}

\makeatletter
\let\c@equation\c@thm

\makeatother
\usepackage{tikz-cd}
\usepackage{mathtools}
\usepackage{adjustbox}
\usepackage{enumitem}
\tikzcdset{scale cd/.style={every label/.append style={scale=#1},
    cells={nodes={scale=#1}}}}
\usepackage{eucal}
\usepackage{mleftright}
\mleftright
\usepackage{mathrsfs}

\newcommand{\liftlabel}[2][0.75pt]{\raisebox{#1}[\height][\dimexpr\depth-#1\relax]{$\scriptstyle #2$}}

\newcommand{\notehelper}[3]{\textcolor{#3}{$\blacksquare$}\marginpar{\ifodd\thepage\raggedright\else\raggedleft\fi\color{#3}\tiny \textbf{#2:} #1}}

\usetikzlibrary{calc}
\usetikzlibrary{decorations.pathmorphing}
\usetikzlibrary{spath3}
\usetikzlibrary{nfold}

\tikzset{curve/.style={settings={#1},to path={(\tikztostart)
    .. controls ($(\tikztostart)!\pv{pos}!(\tikztotarget)!\pv{height}!270:(\tikztotarget)$)
    and ($(\tikztostart)!1-\pv{pos}!(\tikztotarget)!\pv{height}!270:(\tikztotarget)$)
    .. (\tikztotarget)\tikztonodes}},
    settings/.code={\tikzset{quiver/.cd,#1}
        \def\pv##1{\pgfkeysvalueof{/tikz/quiver/##1}}},
    quiver/.cd,pos/.initial=0.35,height/.initial=0}

\tikzset{between/.style n args={2}{/tikz/execute at end to={
    \tikzset{spath/split at keep middle={current}{#1}{#2}}
}}}

\tikzset{tail reversed/.code={\pgfsetarrowsstart{tikzcd to}}}
\tikzset{2tail/.code={\pgfsetarrowsstart{Implies[reversed]}}}
\tikzset{2tail reversed/.code={\pgfsetarrowsstart{Implies}}}
\tikzset{no body/.style={/tikz/dash pattern=on 0 off 1mm}}

\newcommand{\cat}[1]{\mathcal{#1}}
\newcommand{\pr}[1]{\left(#1\right)}
\newcommand{\abs}[1]{\left|#1\right|}

\newcommand{\ps}[2]{{\vphantom{}{}_{#2}}{#1}}

\newcommand{\Cat}{\mathsf{Cat}}
\newcommand{\An}{\mathsf{An}}
\newcommand{\sAn}{\mathsf{sAn}}
\newcommand{\RelCat}{\mathsf{RelCat}}

\newcommand{\sSet}{\mathbf{sSet}}
\newcommand{\sCat}{\mathbf{Cat}_{\mathbf{\Delta}}}

\renewcommand{\lim}{\operatorname{lim}}
\newcommand{\colim}{\operatorname{colim}}
\newcommand{\hocolim}{\operatorname{hocolim}}
\newcommand{\Fun}{\operatorname{Fun}}
\newcommand{\Map}{\operatorname{Map}}

\newcommand{\PSh}{\operatorname{PSh}}
\newcommand{\N}{\operatorname{N}}
\newcommand{\dom}{\operatorname{dom}}
\newcommand{\codom}{\operatorname{codom}}
\newcommand{\Fact}{\operatorname{Fact}}
\newcommand{\cosk}{\operatorname{cosk}}
\newcommand{\Sq}{\operatorname{Sq}}

\newcommand{\Hamm}{\mathrm{Hamm}}
\newcommand{\Zig}{\mathrm{Zig}}
\newcommand{\op}{\mathrm{op}}
\newcommand{\id}{\mathrm{id}}
\newcommand{\ac}{\mathrm{ac}}
\newcommand{\Seg}{\mathrm{Seg}}
\newcommand{\ev}{\mathrm{ev}}
\newcommand{\ob}{\mathrm{ob}}
\newcommand{\pre}{\mathrm{pre}}
\renewcommand{\sp}{\mathrm{sp}}
\newcommand{\comp}{\mathrm{comp}}

\newcommand{\Del}{\mathbf{\Delta}}
\newcommand{\rel}{\mathrm{rel}}
\newcommand{\from}{\colon}

\newcommand{\epi}{\twoheadrightarrow}
\newcommand{\comma}{\downarrow}

\newsavebox{\mspanbox}
\newsavebox{\mcospanbox}
\newsavebox{\msquarebox}
\newsavebox{\marrowbox}
\newsavebox{\mscratchbox}
\newlength{\msqnatht}
\newlength{\marrownatht}
\newlength{\mtargetht}
\newlength{\marrowht}

\AtBeginDocument{%
  \sbox{\mspanbox}{%
    \resizebox{!}{2.5ex}{%
      $\begin{tikzcd}[ampersand replacement=\&]
        \bullet \& \bullet \\
        \bullet
        \arrow[from=1-1, to=1-2]
        \arrow["\sim"', from=1-1, to=2-1]
      \end{tikzcd}$%
    }%
  }%
  \sbox{\mcospanbox}{%
    \resizebox{!}{2.5ex}{%
      $\begin{tikzcd}[ampersand replacement=\&]
		\& \bullet \\
		\bullet \& \bullet
		\arrow["\sim", from=1-2, to=2-2]
		\arrow[from=2-1, to=2-2]
      \end{tikzcd}$%
    }%
  }%
  \sbox{\msquarebox}{%
    \resizebox{!}{2.5ex}{%
      $\begin{tikzcd}[ampersand replacement=\&]
		 \bullet \& \bullet \\
		 \bullet \& \bullet
		 \arrow[from=1-1, to=1-2]
		 \arrow["\sim"', from=1-1, to=2-1]
		 \arrow["\sim", from=1-2, to=2-2]
		 \arrow[from=2-1, to=2-2]
      \end{tikzcd}$%
    }%
  }%
  \sbox{\mscratchbox}{%
    $\begin{tikzcd}[ampersand replacement=\&]
		 \bullet \& \bullet \\
		 \bullet \& \bullet
		 \arrow[from=1-1, to=1-2]
		 \arrow["\sim"', from=1-1, to=2-1]
		 \arrow["\sim", from=1-2, to=2-2]
		 \arrow[from=2-1, to=2-2]
     \end{tikzcd}$%
  }%
  \setlength{\msqnatht}{\ht\mscratchbox}%
  \sbox{\mscratchbox}{%
    $\begin{tikzcd}[ampersand replacement=\&]
       \bullet \& \bullet
       \arrow[from=1-1, to=1-2]
     \end{tikzcd}$%
  }%
  \setlength{\marrownatht}{\ht\mscratchbox}%
  \setlength{\mtargetht}{2.5ex}%
  \pgfmathsetlength{\marrowht}{\mtargetht*(\marrownatht/\msqnatht)}%
  \sbox{\marrowbox}{%
    \resizebox{!}{\marrowht}{%
      $\begin{tikzcd}[ampersand replacement=\&]
         \bullet \& \bullet
         \arrow[from=1-1, to=1-2]
       \end{tikzcd}$%
    }%
  }%
}

\newcommand{\mspan}{\mathord{\raisebox{0.3ex}{\usebox{\mspanbox}}}}
\newcommand{\mcospan}{\mathord{\raisebox{0.3ex}{\usebox{\mcospanbox}}}}
\newcommand{\msquare}{\mathord{\raisebox{0.3ex}{\usebox{\msquarebox}}}}
\newcommand{\marrow}{\mathord{\vcenter{\hbox{\usebox{\marrowbox}}}}}